\documentclass[final]{siamart0216}

\usepackage{lipsum}

\usepackage{amsfonts}
\usepackage{graphicx}
\usepackage{epstopdf}
\usepackage{algorithmic}

\usepackage{amssymb}
\usepackage{amsmath}
\usepackage{subfigure}
\usepackage{hyperref}

\usepackage{braket,amsfonts,amsopn} % <- Preamble
\ifpdf
  \DeclareGraphicsExtensions{.eps,.pdf,.png,.jpg}
\else
  \DeclareGraphicsExtensions{.eps}
\fi

\newcommand{\TheTitle}{Dynamic evaluation of exponential polynomial curves and surfaces via basis transformation}
\newcommand{\TheAuthors}{Xunnian Yang and Jialin Hong}

\headers{Dynamic evaluation via basis transformation}{\TheAuthors}

\title{{\TheTitle}\thanks{This work was supported by National Natural Science Foundation of China grants 11290142. Jialin Hong was also supported by NSFC grants 91530118, 91630312, 91130003.}}

\author{
  Xunnian Yang\thanks{Corresponding author. School of Mathematical Sciences, Zhejiang University, Hangzhou 310027, China
    (\email{yxn@zju.edu.cn}).}
  \and
  Jialin Hong\thanks{LSEC, Academy of Mathematics and Systems Science,
    Chinese Academy of Sciences, Beijing 100190, China (\email{hjl@lsec.cc.ac.cn}).}
}

\DeclareMathOperator{\diag}{diag}

\begin{document}

\maketitle

% REQUIRED
\begin{abstract}
It is shown in "SIAM J. Sci. Comput. 39 (2017):B424-B441" that free-form curves used in computer aided geometric design can usually be represented as the solutions of linear differential systems and points and derivatives on the curves can be evaluated dynamically by solving the differential systems numerically. In this paper we present an even more robust and efficient algorithm for dynamic evaluation of exponential polynomial curves and surfaces. Based on properties that spaces spanned by general exponential polynomials are translation invariant and polynomial spaces are invariant with respect to a linear transformation of the parameter, the transformation matrices between bases with or without translated or linearly transformed parameters are explicitly computed. Points on curves or surfaces with equal or changing parameter steps can then be evaluated dynamically from a start point using a pre-computed matrix. Like former dynamic evaluation algorithms, the newly proposed approach needs only arithmetic operations for evaluating exponential polynomial curves and surfaces. Unlike conventional numerical methods that solve a linear differential system, the new method can give robust and accurate evaluation results for any chosen parameter steps. Basis transformation technique also enables dynamic evaluation of polynomial curves with changing parameter steps using a constant matrix, which reduces time costs significantly than computing each point individually by classical algorithms.

\end{abstract}

% REQUIRED
\begin{keywords}
  curves and surfaces, linear differential operator, exponential polynomial, dynamic evaluation, basis transformation
\end{keywords}

% REQUIRED
\begin{AMS}
65D17, 65D18, 65D25, 65L05
\end{AMS}
%%68U07   Computer-aided design
%%65D17   Computer aided design (modeling of curves and surfaces)
%%65D18   Computer graphics, image analysis, and computational geometry
%%65D25   Numerical differentiation
%%65D30   Numerical integration
%%65D07   Splines
%%65L05   Initial value problems
%%65L06   Multistep, Runge-Kutta and extrapolation methods
%%65L09   Inverse problems

%% main text

%%%%%%%%%%%%%%%%%%%%%%%%%%%%%%%%%%%%%%%%%%%%%%%%%%%%%%%%%%%%%%%%%%%%%%%%%%%%%
%%                                                                          %
%% Section 1                                                                %
%%                                                                          %
%%%%%%%%%%%%%%%%%%%%%%%%%%%%%%%%%%%%%%%%%%%%%%%%%%%%%%%%%%%%%%%%%%%%%%%%%%%%%
\section{Introduction}
\label{Sec:intro}
The exponential polynomials that lie in the null spaces of constant coefficient linear differential operators have nice properties and they have often been used for the construction of curves and surfaces in the fields of CAGD (computer aided geometric design)~\cite{Aldaz2009:BernsteinOperatorsForExpPolynomials,Farin2001CAGDBook,GonsorNeamtu1994:PolarForms}. The frequently used exponential polynomials are polynomials, trigonometric functions, hyperbolic functions or their mixtures. Besides polynomial curves and surfaces, typical curves and surfaces such as ellipses, cycloids, involutes, helices, etc. can be represented by exponential polynomials exactly~\cite{MainarPS01:ShapePreservingAlternatives,Pottmann1994:HelixSplines,Roth2015:DescrptionCurvesInExtChebSpace,Zhang96:C-curve}. By choosing a proper parameter interval, normalized B-bases that are useful for optimal shape design can be obtained from the exponential polynomials~\cite{ChenWang03:Bezier-likeCurves,Mazure2005:ChebyshevSpacesBernsteinBases,RothJ10:ControlPointbasedDescription, Sanchez-Reyes09:PeriodicBezierCurves}. Algebraic trigonometric polynomials can be used to define curves intrinsically or design curves with Pythagorean hodographs~\cite{WuYang2016:IntrinsicHermite,LuciaRomani2014:ATPHcurves,LuciaRomani2018:ATPHspacecurves}.

Many algorithms have been given in the literature to evaluate polynomial curves and surfaces. The de Casteljau algorithm or the rational de Casteljau algorithm can be employed to robustly evaluate single points on B\'{e}zier or rational B\'{e}zier curves~\cite{Farin2001CAGDBook,Pena07:EvaluatingRationalBezierSurface}. The Horner algorithm, the VS algorithm, et al. can be used to evaluate polynomial or B\'{e}zier curves with even lower complexity~\cite{Pena09:RunningRelativeErrorForEvaluation,Bezerra13:FFTBezierEvaluation}. If for rendering or machining purposes, sequences of points have to be evaluated in an efficient way~\cite{ElberCohen96:IsocurveBasedRendering,Gatilov16:NURBSevaluation}. Particularly, the forward differencing approaches have been successfully employed for fast rendering of B\'{e}zier or NURBS (Non-Uniform Rational B-spline) curves and surfaces~\cite{LienSP1987:AdaptiveForwardDifference,LukenCheng96:SurfaceDerivativeForRendering}.

Differently from polynomial curves and surfaces that can be evaluated by arithmetic operations, points on curves and surfaces which are constructed by transcendental functions or mixtures of polynomials and transcendental functions have to be evaluated by inquiring pre-computed special function tables or loading special mathematical libraries. Though this seems feasible for many modern computing machines~\cite{Koren1990:EvaluationElementaryFunctions,Nave1983:ImplementationTranscendentalFunctions}, evaluating general exponential polynomial curves and surfaces by only arithmetic operations without any pre-computed tables or special math library can have its own advantages. Particularly, the speed and efficiency of evaluation play important roles in the fields of CNC machining and interactive rendering.

Recently, we have shown that free-form curves defined in various spaces in CAGD are the solutions of linear differential systems and points and derivatives on the curves can be obtained by solving the linear differential systems numerically~\cite{Yang&Hong2017:DynamicEvaluation}. Particularly, when the parameter step is fixed, points on a free-form curve can be evaluated dynamically by multiplying a pre-computed constant matrix with prior points. Iso-parameter curves on a surface can also be dynamically evaluated by establishing a linear differential system for each iso-parameter curve.
The method is simple and universal and points on polynomial as well as transcendental curves and surfaces can be evaluated with only arithmetic operations. However, the evaluation accuracy varies much when the differential systems have been solved by different numerical methods or the parameter step has been chosen different values. If the constant matrix for dynamic evaluation is given by the exponential of the coefficient matrix of a linear differential system, careful attention should be paid for robust computation of the matrix~\cite{MolerLoan03:19WaysToComputeExpOfMatrix}.

%\subsection{Contributions}
Instead of solving linear differential systems numerically, in this paper we derive the constant matrices for dynamic evaluation of curves and surfaces using basis transformation. This is based on the fact that spaces spanned by the exponential polynomial basis are invariant with respect to the translation of the parameter while curves and surfaces used for CAGD are usually constructed by exponential polynomials. By using identities of exponential polynomials, we compute explicitly the transformation matrices between bases with or without the translation of the parameter. Combined with control points, a constant matrix for evaluating points on an exponential polynomial curve with equal parameter steps is derived. It is also noticed that a polynomial space of degree no more than a given number is even invariant with respect to any linear transformation of the parameter. The matrix for polynomial basis transformation can then be used to evaluate points on polynomial curves with changing parameter steps. Based on basis transformation, a family of surface curves with various iso-parameters can be evaluated dynamically using a single matrix and surface curves with skew parametrization can also be evaluated dynamically with a pre-computed constant matrix.

%\subsection{Outline}
The rest of the paper is structured as follows. In Section~\ref{Sec:Basis transformation for spaces composed of exponential polynomials} we present explicit formulae for computing basis transformation for exponential polynomials. In Section~\ref{Section:Dynamic evaluation of exponential polynomial curves and surfaces} robust algorithm for evaluating points on general exponential polynomial curves with a fixed parameter step, dynamic algorithm for evaluating polynomial curves with changing parameter steps, and dynamic algorithms for evaluating iso-parameter curves on surfaces or surface curves with skew parametrization will be given. Examples and comparisons with some known methods for curve and surface evaluation are given in Section~\ref{Sec:examples}. Section~\ref{Sec:Conclu} concludes the paper with a brief summary of our work.

%%%%%%%%%%%%%%%%%%%%%%%%%%%%%%%%%%%%%%%%%%%%%%%%%%%%%%%%%%%%%%%%%%%%%%%%%%%%%
%%                                                                          %
%% Section 2                                                                %
%%                                                                          %
%%%%%%%%%%%%%%%%%%%%%%%%%%%%%%%%%%%%%%%%%%%%%%%%%%%%%%%%%%%%%%%%%%%%%%%%%%%%%

\section{Basis transformation for spaces composed of exponential polynomials}
\label{Sec:Basis transformation for spaces composed of exponential polynomials}

As parametric curves and surfaces are usually defined by basis functions together with coefficients or control points, a parametric curve or surface can then be evaluated efficiently by exploring distinguished properties of the basis. This section presents explicit transformation formulae for exponential polynomial basis which will be used for robust and efficient curve or surface evaluation in next section.

\subsection{Spaces spanned by exponential polynomials}
\label{Subsection:Spaces_Spanned_by_exponential_polynomials}

Suppose that a linear differential operator $L$ with constant coefficients is given by
\[
L=\left(\frac{d}{dt}-\lambda_0\right)\cdots \left(\frac{d}{dt}-\lambda_n\right),
\]
where $\lambda_i\in \mathbb{C}$, $i=0,1,\ldots,n$, and $\Lambda=\{\lambda_0,\lambda_1,\ldots,\lambda_n\}$ is closed under conjugation. A function $f(t)$ that satisfies $Lf(t)=0$ is referred an exponential polynomial. Let $\Omega$ be the null space of the linear differential operator. From the knowledge of differential equation~\cite{Arnold1992:ODE} we know that $\Omega=\text{span}\{\phi_0(t), \phi_1(t), \ldots, \phi_n(t)\}$, where $\phi_i(t)$, $i=0,1,\ldots,n$, are the basis functions of the space.
Based on the definition of exponential polynomials we have the following proposition.
\begin{proposition}\label{Proposition:null space of LDS}
Suppose that $L$ is a constant coefficient linear differential operator and $\Omega$ is the null space of the operator. Let $h$ be an arbitrary given real number. If function $f(t)$ satisfies $Lf(t)=0$, it yields that $f'(t)\in\Omega$ and $f(t+h)\in\Omega$.
\end{proposition}
Proposition~\ref{Proposition:null space of LDS} states that the null space of a linear differential operator is closed with respect to a differentiation and the space is also invariant with respect to any translation of the parameter.

Before deriving formulae for basis transformation, we present the definition of union or product of two sets of bases. Assume $\Phi_a(t)=\left(a_0(t),a_1(t),\ldots,a_n(t)\right)^T$ and $\Phi_b(t)=\left(b_0(t),b_1(t),\ldots,b_m(t)\right)^T$, where the capital 'T' means the transpose of a vector or matrix. The (ordered) union of $\Phi_a(t)$ and $\Phi_b(t)$ is given by
\begin{equation}\label{Eqn:basis union}
\Phi_a(t) \sqcup \Phi_b(t)=\left(a_0(t),\ldots,a_n(t),b_0(t),\ldots,b_m(t)\right)^T.
\end{equation}
Let $a(t)\Phi_b(t)=\left(a(t)b_0(t),\ldots,a(t)b_m(t)\right)^T$. The product of $\Phi_a(t)$ and $\Phi_b(t)$ is obtained as
\begin{equation}\label{Eqn:bais product}
\Phi_a(t) \otimes \Phi_b(t) = \sqcup_{i=0}^n a_i(t)\Phi_b(t).
\end{equation}
Just like the precedence of '$\times$' over '$+$', we assume the operation '$\otimes$' has precedence over '$\sqcup$'. Thus, $\Phi_1(t)\sqcup\Phi_2(t)\otimes\Phi_3(t)$ has the same meaning as $\Phi_1(t)\sqcup(\Phi_2(t)\otimes\Phi_3(t))$.

\begin{proposition}\label{Proposition:basis of union or product spaces}
Suppose spaces spanned by basis $\Phi_a(t)$ or $\Phi_b(t)$ are closed with respect to a differentiation. Then, spaces spanned by the union $\Phi_a(t)\sqcup\Phi_b(t)$ or by the product $\Phi_a(t)\otimes \Phi_b(t)$ are also closed with respect to a differentiation.
\end{proposition}
\begin{proof}
Because spaces spanned by $\Phi_a(t)$ or $\Phi_b(t)$ are closed with respect to a differentiation, there exist matrix $A$ of order $n+1$ and matrix $B$ of order $m+1$ such that $\Phi_a'(t)=A\Phi_a(t)$ and $\Phi_b'(t)=B\Phi_b(t)$. Then, the derivative of $\Phi_a(t)\sqcup\Phi_b(t)$ can be computed as
\[
\frac{d}{dt}(\Phi_a(t)\sqcup\Phi_b(t))=\diag(A,B)(\Phi_a(t)\sqcup\Phi_b(t)),
\]
where
$
\diag(A,B)=\left(\begin{array}{cc}
                    A & 0 \\
                    0 & B
                    \end{array}
                \right).
$
Let $I_{n+1}$ and $I_{m+1}$ be the identity matrices of order $n+1$ or order $m+1$, respectively. The derivative of $\Phi_a(t)\otimes \Phi_b(t)$ is computed by
\[
\begin{array}{lcl}
\frac{d}{dt}(\Phi_a(t)\otimes\Phi_b(t)) &=& \Phi_a'(t) \otimes \Phi_b(t) + \Phi_a(t) \otimes \Phi_b'(t) \\
                &=& (A \otimes I_{m+1}) (\Phi_a(t) \otimes \Phi_b(t)) + (I_{n+1}\otimes B) (\Phi_a(t) \otimes \Phi_b(t))  \\
                &=& (A \otimes I_{m+1} + I_{n+1}\otimes B) (\Phi_a(t) \otimes \Phi_b(t)).
\end{array}
\]
Therefore, the spaces spanned by $\Phi_a(t)\sqcup\Phi_b(t)$ or by $\Phi_a(t)\otimes \Phi_b(t)$ are also closed under differentiation. This completes the proof.
\end{proof}

In the following we assume that the basis functions are real. Particularly, we assume that the bases are obtained by unions or products of a few elementary basis vectors. Let $U_n(t)=(1,t,t^2,\ldots,t^n)^T$, $V(t)=(\cos(t),\sin(t))^T$ and $W(t)=(\cosh(t),\sinh(t))^T$. The basis functions for free-form curves and surfaces in CAGD can usually be obtained by recursive compositions or tensor products of the elementary bases. Several popular basis vectors for construction of free-form curves and surfaces in CAGD and their elementary decomposition can be found in Table~\ref{Table:Basis vector decomposition}.
\begin{table}[htbp]
\centering
\caption{Basis vectors and their elementary decompositions.}
\begin{tabular}{|c|c|}
\hline
Basis vector & Elementary decomposition \\
\hline
$(1,t,\cos t,\sin t)^T$~\cite{Zhang96:C-curve}           &  $U_1(t)\sqcup V(t)$  \\
\hline
$(1, \cos t,\sin t, \ldots, \cos nt, \sin nt)^T$~\cite{Sanchez-Reyes98:HarmonicRationalBezier}   &$U_0(t)\sqcup V(t)\sqcup\ldots\sqcup V(nt)$  \\
\hline
$(1, \cosh t,\sinh t, \ldots, \cosh nt, \sinh nt)^T$~\cite{ShenW2005:HyperbloicCurves}   &  $U_0(t)\sqcup W(t)\sqcup\ldots\sqcup W(nt)$  \\
\hline
$(1,t,\cos t,\sin t, t\cos t, t\sin t)^T$~\cite{MainarPS01:ShapePreservingAlternatives}  &  $U_1(t)\sqcup U_1(t)\otimes V(t)$  \\

\hline
$(1,t,\ldots,t^{n-2},\cos t,\sin t)^T$~\cite{ChenWang03:Bezier-likeCurves}  &  $U_{n-2}(t)\sqcup V(t)$ \\

\hline
$(1,t,\ldots,t^{n-2},\cosh t,\sinh t)^T$~\cite{LiYajuan2005:twokindsofB-basis}  &  $U_{n-2}(t)\sqcup W(t)$  \\

\hline
$(1,\cosh t, \sinh t,\cos t,\sin t)^T$~\cite{BrilleaudM12:MixedHyperbolicTrigonometric}  &  $U_0(t)\sqcup W(t)\sqcup V(t)$ \\

\hline
$(1,t,\ldots,t^{n-5},\cosh t,\sinh t, \cos t,\sin t)^T$~\cite{XuGangW07:AHT-Bezier}  &  $U_{n-5}(t)\sqcup W(t)\sqcup V(t)$  \\

\hline
$(1,\cos t,\sin t, t\cos t, t\sin t,\ldots, t^n\cos t, t^n\sin t)^T$~\cite{WuYang2016:IntrinsicHermite}  &  $U_0(t)\sqcup U_n(t)\otimes V(t)$  \\

\hline
\end{tabular}\label{Table:Basis vector decomposition}
\end{table}

\subsection{Transformation of general exponential polynomial basis with parameter translation}
\label{Subsection:basis transformation with parameter translation}

In this subsection we show that the space spanned by general exponential polynomial basis is invariant with respect to a translation of the parameter and a simple and robust method for computing the transformation matrix between different bases will be presented.
\begin{proposition}\label{Proposition:space with tranlated basis}
Suppose $\Omega=\text{span}\{\phi_0(t),\phi_1(t),\ldots,\phi_n(t)\}$ is closed with respect to a differentiation. Let $h$ be an arbitrary real number. It yields that $\Omega=\text{span}\{\phi_0(t+h),\phi_1(t+h),\ldots,\phi_n(t+h)\}$. Meanwhile, there exists a matrix $C_h$ such that $\Phi(t+h)=C_h \Phi(t)$.
\end{proposition}
\begin{proof}
Because $\phi_i'(t)\in\Omega$, $i=0,1,\ldots,n$, there exists a matrix $A$ such that $\Phi(t)$ satisfies a linear differential system
\begin{equation}\label{Eqn:linear system for basis}
\left\{
\begin{array}{lcl}
\Phi'(t) & = & A\Phi(t), \ \ \ t\in \mathbb{R}. \\
\Phi(t_0)& = & \Phi_0,
\end{array}
\right.
\end{equation}
From Equation~(\ref{Eqn:linear system for basis}), $\Phi(t)$ can be represented as $\Phi(t)=e^{A(t-t_0)}\Phi_0$. Therefore, we have $\Phi(t+h)=e^{Ah}\Phi(t)$. Let $C_h=e^{Ah}$. Since $\text{det}(C_h)\neq 0$ and $\Phi(t)=C_h^{-1}\Phi(t+h)$, we have $\Omega=\text{span}\{\phi_0(t+h),\phi_1(t+h),\ldots,\phi_n(t+h)\}$. This proves the proposition.
\end{proof}

\begin{proposition}\label{Proposition: invariant surface basis space}
Suppose $\Omega=\textrm{span}\{\phi_0(u,v),\phi_1(u,v),\ldots,\phi_L(u,v)\}$ and $\frac{\partial\phi_l(u,v)}{\partial u}\in\Omega$, $\frac{\partial\phi_l(u,v)}{\partial v}\in\Omega$, $l=0,1,\ldots,L$. The space $\Omega$ is invariant when parameter $u$ or parameter $v$ or both of the two parameters have been translated.
\end{proposition}
\begin{proof}
We first prove that the space $\Omega$ is invariant when the parameter $u$ or $v$ has been translated. Let $\Phi(u,v)=(\phi_0(u,v),\phi_1(u,v),\ldots,\phi_L(u,v))^T$. Because $\frac{\partial\phi_l(u,v)}{\partial u}\in\Omega$, $l=0,1,\ldots,L$, there exists matrix $A_1$ such that the basis vector $\Phi(u,v)$ satisfy $\frac{\partial\Phi(u,v)}{\partial u} = A_1\Phi(u,v)$. Therefore, we have $\Phi(u,v)=e^{A_1(u-u_0)}\Phi(u_0,v)$ for any selected real number $u_0$. From this expression of $\Phi(u,v)$, we have $\Phi(u+h_1,v)=e^{A_1 h_1}\Phi(u,v)$. By the same reason, we have $\frac{\partial\Phi(u,v)}{\partial v} = A_2\Phi(u,v)$ and $\Phi(u,v+h_2)=e^{A_2 h_2}\Phi(u,v)$. Denote $C_u^{h_1}=e^{A_1 h_1}$ and $C_v^{h_2}=e^{A_2 h_2}$. Because matrices $C_u^{h_1}$ and $C_v^{h_2}$ are non-singular, it implies that both $\Phi(u+h_1,v)$ and $\Phi(u,v+h_2)$ are basis vectors of the space $\Omega$.

Now we prove that the space $\Omega$ is invariant when both parameters $u$ and $v$ have been translated. In fact, $\Phi(u+h_1,v+h_2)=C_u^{h_1}\Phi(u,v+h_2)=C_u^{h_1}C_v^{h_2}\Phi(u,v)$. Let $C_{u,v}^{h_1,h_2}=C_u^{h_1}C_v^{h_2}$. Since $\det(C_{u,v}^{h_1,h_2})=\det(C_u^{h_1})\det(C_v^{h_2})\neq 0$, we conclude that $\Phi(u+h_1,v+h_2)$ is also the basis vector of space $\Omega$.
\end{proof}

Though the transformation matrices between bases with or without translation of the parameters are defined by exponentials of constant matrices, accurate and efficient evaluation of exponentials of matrices is not a trivial task~\cite{MolerLoan03:19WaysToComputeExpOfMatrix}. We propose to compute the transformation matrices for exponential polynomial bases with translated parameters directly. Particularly, we derive the transformation matrices for polynomials, trigonometric functions or hyperbolic functions using the identities of the functions first and then compute the transformation matrices for even more general basis through their elementary decompositions.

Assume $U_n(t)$, $V(t)$ and $W(t)$ are the basis vectors as given in Subsection~\ref{Subsection:Spaces_Spanned_by_exponential_polynomials}. With simple computation, the transformation for basis vector $U_n(t)$ can be obtained as $U_n(t+h)=M_{U_n}^h U_n(t)$, where
\[
M_{U_n}^h=\left(\begin{array}{ccccc}
                             1 & 0 & 0 & \cdots & 0 \\
                             h & 1 & 0  &  \cdots & 0 \\
                             h^2 & 2h & 1 & \cdots & 0 \\
                             \vdots & \vdots & \vdots & \ddots & \vdots \\
                             h^n & \binom{n}{1}h^{n-1} & \binom{n}{2}h^{n-2} & \cdots & 1
                           \end{array}
                        \right).
\]
Similarly, the transformation matrices for $V(t)$ and $W(t)$ are
\[
M_V^h=\left(\begin{array}{cc}
                    \cos h & -\sin h \\
                    \sin h &  \cos h
                    \end{array}
                \right)
\]
and
\[
M_W^h=\left(\begin{array}{cc}
                    \cosh h &  \sinh h \\
                    \sinh h &  \cosh h
                    \end{array}
                \right),
\]
respectively. It yields that $V(t+h)=M_V^h V(t)$ and $W(t+h)=M_W^h W(t)$.

For a small parameter step $h$, the values of $\cos h$, $\sin h$, $\cosh h$ and $\sinh h$ can be computed efficiently with only arithmetic operations by Taylor expansion:
\[
\begin{array}{lll}
\sin h &=& h-\frac{h^3}{3!}+\frac{h^5}{5!}-\frac{h^7}{7!}+\frac{h^9}{9!}-\frac{h^{11}}{11!}+\frac{h^{13}}{13!}-\frac{h^{15}}{15!} \\
\cos h &=& 1-\frac{h^2}{2!}+\frac{h^4}{4!}-\frac{h^6}{6!}+\frac{h^8}{8!}-\frac{h^{10}}{10!}+\frac{h^{12}}{12!}-\frac{h^{14}}{14!}-\frac{h^{16}}{16!}\\
\sinh h &=& h+\frac{h^3}{3!}+\frac{h^5}{5!}+\frac{h^7}{7!}+\frac{h^9}{9!}+\frac{h^{11}}{11!}+\frac{h^{13}}{13!}+\frac{h^{15}}{15!} \\
\cosh h &=& 1+\frac{h^2}{2!}+\frac{h^4}{4!}+\frac{h^6}{6!}+\frac{h^8}{8!}+\frac{h^{10}}{10!}+\frac{h^{12}}{12!}+\frac{h^{14}}{14!}+\frac{h^{16}}{16!}
\end{array}
\]
The above expressions can be evaluated by Horner algorithm in practice.
From the definition of $M_V^h$ and $M_W^h$ we know that $M_V^h=(M_V^{h/K})^K$ and $M_W^h=(M_W^{h/K})^K$. If the parameter step $h$ is larger than a threshold, for example $0.1$, we can choose a proper integer $K$ and compute the elements of matrix $M_V^{h/K}$ or $M_W^{h/K}$ first, and then compute the matrix $M_V^h$ or $M_W^h$ by matrix multiplication. Similarly, transformation matrices with other parameter steps can be computed by $M^{2h}=(M^h)^2$, $M^{3h}=(M^h)^3$, etc.

Assume the transformation matrices for $\Phi_1(t)$ or $\Phi_2(t)$ are $M_1$ and $M_2$, respectively. The transformation matrix for $\Phi_1(t) \sqcup \Phi_2(t)$ can be computed as
\begin{equation}\label{Eqn:transform of basis union}
\Phi_1(t+h) \sqcup \Phi_2(t+h)=\diag(M_1,M_2) (\Phi_1(t) \sqcup \Phi_2(t)).
\end{equation}
The transformation matrix for $\Phi_1(t) \otimes \Phi_2(t)$ is obtained as follows
\begin{equation}\label{Eqn:transform of basis product}
\Phi_1(t+h) \otimes \Phi_2(t+h)= (M_1 \otimes M_2) (\Phi_1(t) \otimes \Phi_2(t)),
\end{equation}
where $M_1 \otimes M_2$ is also known as the Kronecker product of two matrices.

Similar to the product of two univariate bases, the transformation matrix for the tensor product basis of a surface can be computed easily. Suppose that $\Phi(u,v)=\Phi_1(u)\otimes\Phi_2(v)$, $\Phi_1(u+h_1)=M_1\Phi_1(u)$ and $\Phi_2(v+h_2)=M_2\Phi_2(v)$, the transformation matrix for $\Phi(u,v)$ is computed by
\begin{equation}\label{Eqn:transform of tensor product basis}
\begin{array}{lcl}
\Phi(u+h_1,v+h_2) &=& \Phi_1(u+h_1) \otimes \Phi_2(v+h_2) \\
                &=& (M_1 \otimes M_2) (\Phi_1(u) \otimes \Phi_2(v)) \\
                &=& (M_1 \otimes M_2) \Phi(u,v).
\end{array}
\end{equation}
By Equations (\ref{Eqn:transform of basis union}), (\ref{Eqn:transform of basis product}) and (\ref{Eqn:transform of tensor product basis}), the transformation matrices for even more complicated basis will be computed accurately and robustly.

\subsection{Transformation of polynomial basis with linear transformation of the parameter}
\label{Subsection:Linear transformation of polynomial basis}

In this subsection we show that space of polynomials of degree no more than a given number is invariant with respect to a linear transformation of the parameter. Particularly, the transformation matrix for Bernstein basis will be given.

\begin{proposition}\label{Proposition: power basis invariant space}
Let $\mathbb{P}_n(t)=\text{span}\{1,t,t^2,\ldots,t^n\}$. The space $\mathbb{P}_n(t)$ is invariant with respect to a linear transformation of the parameter.
\end{proposition}
\begin{proof}
Let $U_n(t)=(1,t,t^2,\ldots,t^n)^T$. Assume $a_0\neq 0$ and $a_1$ are real numbers. Replacing $t$ within $U_n(t)$ by $a_0t+a_1$, we have
$U_n(a_0 t+a_1)=C_A U_n(t)$, where
\[
C_A=\left(\begin{array}{ccccc}
            1      &    0    &   0    &  \ldots &  0  \\
            a_1    &  a_0    &   0    &  \ldots &  0  \\
            a_1^2  & 2a_0a_1 &  a_0^2 &  \ldots &  0  \\
            \vdots & \vdots  & \vdots &  \ddots & \vdots  \\
            a_1^n  & \binom{n}{1}a_0a_1^{n-1} & \binom{n}{2}a_0^2a_1^{n-2} & \ldots & a_0^n
            \end{array}
    \right).
\]
Because $\det(C_A)=a_0^{n(n+1)/2}\neq 0$, we have $U_n(t)=C_A^{-1}U_n(a_0t+a_1)$. Therefore, $U_n(a_0t+a_1)$ is another set of basis of the space $\mathbb{P}_n(t)$. The proposition is proven.
\end{proof}

Besides power basis, another popular basis used for polynomial curve and surface modeling is Bernstein basis. We derive transformation matrix for Bernstein basis under a linear transformation of the parameter. The transformation matrix will be used for dynamic evaluation of B\'{e}zier curves and surfaces with constant or changing parameter steps.

\begin{proposition}\label{Proposition: Bernstein basis invariant space}
Assume $\Phi_B(t)=(B_{0,n}(t),B_{1,n}(t),\ldots,B_{n,n}(t))^T$, where $B_{i,n}(t)=\frac{n!}{i!(n-i)!}t^i(1-t)^{n-i}$, $i=0,1,\ldots,n$, are Bernstein basis functions. Assume $a\neq b$ are real numbers. Then the basis vector satisfies
\begin{equation}
\label{Eqn:Bernstein basis transform}
\Phi_B((1-t)a+tb)=C_B\Phi_B(t),
\end{equation}
where
\[
C_B=\left(\begin{array}{cccc}
            c_{00} & c_{01} & \ldots & c_{0n}  \\
            c_{10} & c_{11} & \ldots & c_{1n}  \\
            \vdots & \vdots & \ddots & \vdots  \\
            c_{n0} & c_{n1} & \ldots & c_{nn}
            \end{array}
    \right)
\]
and $c_{kl}=\sum\limits^{i+j=k}_{0\leq i\leq l,0\leq j\leq n-l} B_{i,l}(b)B_{j,n-l}(a)$, $0\leq k,l\leq n$.
\end{proposition}
\begin{proof}
Let $\mathbf{e}_0=(1,0,0,\ldots,0)^T$, $\mathbf{e}_1=(0,1,0,\ldots,0)^T$, $\ldots$, $\mathbf{e}_n=(0,0,0,\ldots,1)^T$. The basis vector $\Phi_B(t)$ can be represented as $\Phi_B(t)=\sum_{i=0}^n \mathbf{e}_i B_{i,n}(t)$. Assume $I$ and $E$ are identity or shift operators which satisfy $I\mathbf{e}_i=\mathbf{e}_i$ and $E\mathbf{e}_i=\mathbf{e}_{i+1}$. The basis vector can be reformulated as $\Phi_B(t)=[(1-t)I+tE]^n \mathbf{e}_0$. Then we have
\[
\begin{array}{lcl}
\Phi_B((1-t)a+tb) &=&\{[1-(1-t)a-tb]I+[(1-t)a+tb]E\}^n \mathbf{e}_0 \\
                &=&\{[(1-a)I+aE](1-t)+[(1-b)I+b E]t\}^n \mathbf{e}_0 \\
                &=&\sum_{l=0}^n [(1-b)I + bE]^l [(1-a)I + a E]^{n-l}\mathbf{e}_0 B_{l,n}(t) \\
                &=&\sum_{l=0}^n \mathbf{q}_l B_{l,n}(t),
\end{array}
\]
where
\[
\begin{array}{lcl}
\mathbf{q}_l &=&[(1-b)I + bE]^l [(1-a)I + a E]^{n-l}\mathbf{e}_0 \\
    &=&\sum_{i=0}^l E^i B_{i,l}(b) \sum_{j=0}^{n-l} E^j B_{j,n-l}(a) \mathbf{e}_0 \\
    &=&\sum_{k=0}^n\sum^{i+j=k}_{0\leq i\leq l,0\leq j\leq n-l} B_{i,l}(b)B_{j,n-l}(a) E^k \mathbf{e}_0\\
    &=&\sum_{k=0}^n c_{kl}\mathbf{e}_k.
\end{array}
\]
As $\mathbf{q}_l=\sum_{k=0}^n c_{kl} \mathbf{e}_k=(c_{0l},c_{1l},\ldots,c_{nl})^T$, $l=0,1,\ldots,n$, representing $\Phi_B((1-t)a+tb)$ in matrix form, we have
\[
\begin{array}{lcl}
\Phi_B((1-t)a+tb)&=&(\begin{array}{cccc}
                        \mathbf{q}_0 & \mathbf{q}_1 & \ldots & \mathbf{q}_n
                        \end{array}
                        )
                        \left(\begin{array}{c}
                             B_{0,n}(t) \\
                             B_{1,n}(t) \\
                             \vdots \\
                             B_{n,n}(t)
                           \end{array}
                        \right) \\
                    &=&
                        \left(\begin{array}{cccc}
                            c_{00} & c_{01} & \ldots & c_{0n}  \\
                            c_{10} & c_{11} & \ldots & c_{1n}  \\
                            \vdots & \vdots & \ddots & \vdots  \\
                            c_{n0} & c_{n1} & \ldots & c_{nn}
                            \end{array}
                        \right)
                        \left(\begin{array}{c}
                             B_{0,n}(t) \\
                             B_{1,n}(t) \\
                             \vdots \\
                             B_{n,n}(t)
                           \end{array}
                        \right) \\
                    &=&C_B \Phi_B(t).
\end{array}
\]
This proves the proposition.
\end{proof}

%%%%%%%%%%%%%%%%%%%%%%%%%%%%%%%%%%%%%%%%%%%%%%%%%%%%%%%%%%%%%%%%%%%%%%%%%%%%%
%%                                                                          %
%% Section 3                                                                %
%%                                                                          %
%%%%%%%%%%%%%%%%%%%%%%%%%%%%%%%%%%%%%%%%%%%%%%%%%%%%%%%%%%%%%%%%%%%%%%%%%%%%%

\section{Dynamic evaluation of exponential polynomial curves and surfaces}
\label{Section:Dynamic evaluation of exponential polynomial curves and surfaces}

In this section we show that curves and surfaces defined by basis and control points can be evaluated dynamically by applying the basis transformation recursively. If a linear differential system is available, the derivatives of the curve or surface at the evaluated points can be obtained simultaneously.

\subsection{Dynamic evaluation of exponential polynomial curves}
\label{Subsection:evaluation of exponential polynomial curves}

Suppose a free-form curve is defined by $P(t)=\sum_{i=0}^n P_i\phi_i(t)$, where $P_i\in \mathbb{R}^d$, $i=0,1,\ldots,n$, are the control points. Let $\Phi(t)=(\phi_0(t),\phi_1(t),\ldots,\phi_n(t))^T$. The curve can be represented as $P(t)=M_P\Phi(t)$, where $M_P=(P_0,P_1,\ldots,P_n)$ and $\Phi(t)$ is also referred as the normal curve in $\mathbb{R}^{n+1}$~\cite{Goldman2009:IntroductionToGraphics&Modeling}. If $d<n+1$, we first lift all control points in space $\mathbb{R}^{n+1}$ by adding additional coordinates as that presented in~\cite{Yang&Hong2017:DynamicEvaluation}. Assume the lifted curve is $X(t)=\sum_{i=0}^n X_i\phi_i(t)$, where $X_i\in\mathbb{R}^{n+1}$, $i=0,1,\ldots,n$. When a point on $X(t)$ has been evaluated, the point on $P(t)$ will be obtained just by choosing the first few coordinates.

Suppose that the space $\Omega=\text{span}\{\phi_0(t),\phi_1(t),\ldots,\phi_n(t)\}$ is closed under a differentiation and the matrix $M_X=(X_0,X_1,\ldots,X_n)$ is nonsingular. The curve $X(t)$ is formulated as $X(t)=M_X\Phi(t)$. From Proposition~\ref{Proposition:space with tranlated basis} we know that the basis vector satisfies $\Phi(t+h)=C_h\Phi(t)$. Therefore, all points on curve $X(t)$ satisfy $X(t+h)=M_XC_h\Phi(t)$. Substituting $\Phi(t)=M_X^{-1}X(t)$, we have $X(t+h)=M_hX(t)$, where $M_h=M_XC_hM_X^{-1}$. Suppose an initial point at parameter $t=t_0$ is known, points with a constant time step will be computed dynamically as follows
\begin{equation}\label{Eqn:dynamic evaluation for general curve X(t)}
\left\{
\begin{array}{lcl}
X(t_0) & = & X_{orig},\\
X(t_i) & = & M_h X(t_{i-1}), \ \ \ i=1,2,\ldots
\end{array}
\right.
\end{equation}
where $t_i=t_0+hi$.
From Equation (\ref{Eqn:linear system for basis}) we know that $\Phi'(t)=A\Phi(t)$. Then we have $X'(t) = A_c X(t)$, where $A_c = M_X A M_X^{-1}$. The derivatives at the evaluated points are obtained as $X'(t_i)=A_c X(t_i)$, $X''(t_i)=A_c X'(t_i)=A_c^2 X(t_i)$, etc.

As discussed in Subsection~\ref{Subsection:Linear transformation of polynomial basis} the polynomial spaces are invariant under a linear transformation of the parameter. Points with changing parameter steps on a polynomial curve can be evaluated by using a fixed iteration matrix. Assume $\Phi_B(t)$ be the basis vector as given in Proposition~\ref{Proposition: Bernstein basis invariant space}. By applying a linear transformation that maps interval $[0,1]$ to $[a,b]$, the basis vector becomes $\Phi_B((1-t)a+tb)=C_B \Phi_B(t)$. A polynomial curve $X(t)=M_X\Phi_B(t)$ satisfies $X((1-t)a+tb)=M_B X(t)$, where $M_B=M_X C_B M_X^{-1}$. Starting from an initial point $X(t_0)$, a sequence of points on curve $X(t)$ will be computed by
\begin{equation}
\label{Eqn:dynamic evaluation of Bezier curve}
X(t_i)=M_B X(t_{i-1}), \ \ \ i=1,2,\ldots
\end{equation}
where
\[
\begin{array}{lcl}
t_i &=&a+(b-a)t_{i-1} \\
    &=&a+a(b-a)+\cdots+a(b-a)^{i-1}+t_0(b-a)^i \\
    &=&\left\{
        \begin{array}{ll}
        t_0+ia & \text{if} \ (b-a)=1\\
        a\frac{1-(b-a)^i}{1-(b-a)}+t_0(b-a)^i & \text{otherwise}.
        \end{array}
        \right.
\end{array}
\]
A linear differential system that represents the B\'{e}zier curves has been given in~\cite{Yang&Hong2017:DynamicEvaluation}. By the differential system, the derivatives at any evaluated point on a B\'{e}zier curve can be obtained just by multiplying the coefficient matrix with the point.

\subsection{Dynamic evaluation of exponential polynomial surfaces}
\label{Subsec:evaluation of exponential polynomial surfaces}

Similar to free-form curves, a bivariate surface can also be reformulated as the solution to a linear differential system when the space spanned by the basis functions is closed with respect to the partial differentials. Iso-parameter curves and surface curves with skew parametrization can be computed dynamically via basis transformation.

$\bullet$ \textbf{Dynamic evaluation of iso-curves of free-form surfaces.}
Suppose a surface is given by $X(u,v)=\sum_{l=0}^L X_l \phi_l(u,v)$, where $X_l\in \mathbb{R}^{L+1}$, $l=0,1,\ldots,L$, are the control points and $\phi_l(u,v)$, $l=0,1,\ldots,L$, are the basis or blending functions. Assume the matrix $M_X=(X_0,X_1,\ldots,X_L)$ is nonsingular and the space $\Omega=\text{span}\{\phi_0(u,v),\phi_1(u,v),\ldots,\phi_L(u,v)\}$ is closed with respect to the partial differentiations $\frac{\partial}{\partial u}$ and $\frac{\partial}{\partial v}$. Let $\Phi(u,v)=(\phi_0(u,v),\phi_1(u,v),\ldots,\phi_L(u,v))^T$. The surface is reformulated as $X(u,v)=M_X \Phi(u,v)$. From Proposition~\ref{Proposition: invariant surface basis space} we know that $\Phi(u+h_1,v)=C_u^{h_1}\Phi(u,v)$ and $\Phi(u,v+h_2)=C_v^{h_2}\Phi(u,v)$. Let $M_u^{h_1} = M_X C_u^{h_1} M_X^{-1}$ and $M_v^{h_2} = M_X C_v^{h_2} M_X^{-1}$. The iso-parameter curves or points on all $u$-curves or all $v$-curves can be dynamically evaluated by
\begin{equation}
\label{Eqn: evaluate X(u+h1,v)}
X(u+h_1,v)= M_u^{h_1} X(u,v)
\end{equation}
or
\begin{equation}
\label{Eqn: evaluate X(u,v+h2)}
X(u,v+h_2)= M_v^{h_2} X(u,v).
\end{equation}

From the proof of Proposition~\ref{Proposition: invariant surface basis space} we also know that $\frac{\partial\Phi(u,v)}{\partial u}=A_1\Phi(u,v)$ and $\frac{\partial\Phi(u,v)}{\partial v}=A_2\Phi(u,v)$. Therefore, the derivatives of $X(u,v)$ can be computed by
\[
\begin{array}{lcl}
\frac{\partial X(u,v)}{\partial u} &=& A_u X(u,v),  \\
\frac{\partial X(u,v)}{\partial v} &=& A_v X(u,v),
\end{array}
\]
where $A_u=M_X A_1 M_X^{-1}$ and $A_v=M_X A_2 M_X^{-1}$. Besides the first order derivatives, higher order derivatives can also be computed by multiplying these two matrices repeatedly. For example, $\frac{\partial^2 X(u,v)}{\partial u^2}=A_u^2 X(u,v)$, $\frac{\partial^2 X(u,v)}{\partial u\partial v}=A_v A_u X(u,v)$.

$\bullet$ \textbf{Dynamic evaluation of surface curves with skew parametrization.}
In addition to the iso-parameter curves, curves with skew parametrization on a surface can satisfy differential equations and can be evaluated dynamically too. Assume $\gamma(t)=(u(t),v(t))$ is a tangent smooth curve in the parameter domain. A surface curve is obtained as $Q(t)=X(u(t),v(t))$.
The derivative of the surface curve is
\[
\begin{array}{lcl}
Q'(t)&=&\frac{dX(u(t),v(t))}{dt} \\
&=&u'(t)\frac{\partial X(u,v)}{\partial u}+v'(t)\frac{\partial X(u,v)}{\partial v} \\
&=&M_X[u'(t)A_1+v'(t)A_2]\Phi(u(t),v(t)).
\end{array}
\]
Substituting $\Phi(u(t),v(t))=M_X^{-1}X(u(t),v(t))$ into above equation, we have %obtain a linear differential system
\begin{equation}\label{Eqn:linear system for surface curve Q(t)}
\left\{
\begin{array}{lcl}
Q'(t)    & = & A_\gamma Q(t), \ \ \ t\in[\alpha,\beta], \\
Q(\alpha)& = & X(u(\alpha),v(\alpha)),
\end{array}
\right.
\end{equation}
where $A_\gamma=M_X[u'(t)A_1+v'(t)A_2]M_X^{-1}$. In particular, if $u(t)$ and $v(t)$ are linear functions, i.e., $u(t)=u_0+\delta(t-\alpha)$ and $v(t)=v_0+\eta(t-\alpha)$, it yields that $u'(t)=\delta$, $v'(t)=\eta$, and $A_\gamma$ is a constant matrix.

Assume $u(t)$ and $v(t)$ are linear functions, we evaluate a sequence of points on curve $Q(t)=X(u(t),v(t))$ by basis transformation. Suppose that the point $Q(t_i)=X(u_i,v_i)$ is known, we compute $Q(t_i+h)$ as follows
\begin{equation}\label{Eqn:forward mapping surface curve X(t)}
\begin{array}{lcl}
Q(t_i+h)&=& X(u_i+\delta h,v_i+\eta h)    \\
        &=& M_X \Phi(u_i+\delta h,v_i+\eta h) \\
        &=& M_X C_{u,v}^{\delta h,\eta h} \Phi(u_i,v_i) \\
        &=& M_X C_{u,v}^{\delta h,\eta h} M_X^{-1} Q(t_i) \\
        &=& M_{u,v}^{\delta h,\eta h}Q(t_i),
\end{array}
\end{equation}
where $C_{u,v}^{\delta h,\eta h}=C_{u}^{\delta h}C_{v}^{\eta h}$ is the basis transformation matrix as defined in Proposition~\ref{Proposition: invariant surface basis space}. When the matrices $C_{u,v}^{\delta h,\eta h}$ and $M_{u,v}^{\delta h,\eta h}=M_X C_{u,v}^{\delta h,\eta h} M_X^{-1}$ have been obtained, points and derivatives of the curve $Q(t)$ will be evaluated dynamically by Equations (\ref{Eqn:forward mapping surface curve X(t)}) and (\ref{Eqn:linear system for surface curve Q(t)}).

\subsection{Evaluation of curves and surfaces with combined parameter steps}
\label{Subsection:adaptive evaluation}

For efficiency of rendering and machining, points on a curve or surface may be evaluated with non-constant or adaptive parameter steps. For polynomial curves and surfaces, Equation (\ref{Eqn:dynamic evaluation of Bezier curve}) can be employed to evaluate points with changing parameter steps. For general exponential curves and surfaces one may compute transformation matrices with different parameter steps first and then compute points on curves or surfaces using combinations of these transformation matrices.

From Proposition~\ref{Proposition:space with tranlated basis} we know that the transformation of a basis vector satisfies $\Phi(t+h_1+h_2)=C_{h_1+h_2}\Phi(t)=C_{h_1}C_{h_2}\Phi(t)=C_{h_2}C_{h_1}\Phi(t)$. It is also verified that the transformation matrices given by Equation (\ref{Eqn:dynamic evaluation for general curve X(t)}) satisfies $M_{h_1+h_2}=M_{h_1}M_{h_2}=M_{h_2}M_{h_1}$. Then, if one or more points should be added within a curve segment, we can begin with any point on the curve and compute additional points using transformation matrices with smaller parameter steps.

A surface $X(u,v)=M_X \Phi(u,v)$ can be evaluated along $u$-curves, $v$-curves, or curves with skew parametrization using different transformation matrices. From Proposition~\ref{Proposition: invariant surface basis space} we know that the transformation matrices for $\Phi(u,v)$ satisfy $C_{u,v}^{h_1,h_2}=C_u^{h_1}C_v^{h_2}=C_v^{h_2}C_u^{h_1}$. Then the transformation matrices for surface points satisfy $M_{u,v}^{h_1,h_2}=M_u^{h_1}M_v^{h_2}=M_v^{h_2}M_u^{h_1}$. This implies that $X(u+h_1,v+h_2)$ can be computed by a transformation from $X(u,v)$ directly or through intermediate points like $X(u+h_1,v)$ or $X(u,v+h_2)$. Using combinations of matrices $M_u^{h_1}$, $M_v^{h_2}$ and $M_{u,v}^{h_1,h_2}$, one can compute curves with even more complicated parameter steps on a surface.

%%%%%%%%%%%%%%%%%%%%%%%%%%%%%%%%%%%%%%%%%%%%%%%%%%%%%%%%%%%%%%%%%%%%%%%%%%%%%
%%                                                                          %
%% Section 4                                                                %
%%                                                                          %
%%%%%%%%%%%%%%%%%%%%%%%%%%%%%%%%%%%%%%%%%%%%%%%%%%%%%%%%%%%%%%%%%%%%%%%%%%%%%

\section{Examples and comparisons}
\label{Sec:examples}
The proposed algorithm for curve and surface evaluation was implemented using C++ on a laptop with Intel(R) Core(TM) i7-4910MQ CPU@2.90GHz 2.89GHz and 8G RAM. All numbers are represented in double precision. Comparisons between the proposed method and some known algorithms will be given.

\textbf{Example 1.}
In the first example we evaluate an intrinsically defined planar curve
\begin{eqnarray}
\textbf{r}(\theta)=\Bigg(\begin{array}{c}
             x(\theta) \\
             y(\theta)
           \end{array}\Bigg)
         =\Bigg(\begin{array}{c}
             \int_0^\theta \rho(t)\cos tdt \\
             \int_0^\theta \rho(t)\sin tdt
           \end{array}\Bigg)
\label{Eqn:integration_curve}
\end{eqnarray}
where $\rho(t)$ represents the curvature radius of a planar curve and $\theta$ is the angle between the tangent direction of the curve and the positive direction of $x$-axis. Just as in~\cite{WuYang2016:IntrinsicHermite}, we choose $\rho(t)=0.001t^3-0.06t^2+1.5t+0.4$. The Cartesian coordinates of the curve are given by
\[
\begin{array}{lcl}
\textbf{r}(\theta)&=&
    \Big(\begin{array}{c} -1.494 \\ 0.52   \end{array}\Big)
    +\Big(\begin{array}{c} 1.494  \\ -0.52  \end{array}\Big)\cos\theta
    +\Big(\begin{array}{c} 0.52   \\ 1.494  \end{array}\Big)\sin\theta  \\
&&  +\Big(\begin{array}{c} -0.12 \\ -1.494 \end{array}\Big)\theta\cos\theta
    +\Big(\begin{array}{c} 1.494 \\ -0.12  \end{array}\Big)\theta\sin\theta
    +\Big(\begin{array}{c} 0.003 \\ 0.06   \end{array}\Big)\theta^2\cos\theta \\
&&  +\Big(\begin{array}{c} -0.06 \\ 0.003  \end{array}\Big)\theta^2\sin\theta
    +\Big(\begin{array}{c}  0 \\ -0.001    \end{array}\Big)\theta^3\cos\theta
    +\Big(\begin{array}{c}  0.001\\ 0      \end{array}\Big)\theta^3\sin\theta.
\end{array}
\]

Just like that presented in~\cite{Yang&Hong2017:DynamicEvaluation}, we lift the curve from $\mathbb{R}^2$ to $\mathbb{R}^9$. Assume the lifted coefficient vectors are $X_i\in\mathbb{R}^9$, $i=0,1,\ldots,8$. The lifted curve is obtained as $X(\theta)=M_X\Phi(\theta)$, where $M_X=(X_0,X_1,\ldots,X_8)$ and $\Phi(\theta)$ is the corresponding basis vector. From Section~\ref{Sec:Basis transformation for spaces composed of exponential polynomials} we know that $\Phi(\theta)=U_0(\theta)\sqcup U_3(\theta)\otimes V(\theta)$, where $U_n(\theta)$ and $V(\theta)$ are elementary bases as defined in Subsection~\ref{Subsection:Spaces_Spanned_by_exponential_polynomials}. Based on Equations (\ref{Eqn:transform of basis union}) and (\ref{Eqn:transform of basis product}), the transformation of the basis vector with a translated parameter step is obtained as $\Phi(\theta+h)=C_h \Phi(\theta)$, where $C_h=\diag(M_{U_0}^h, M_{U_3}^h \otimes M_V^h)$. Furthermore, we compute $M_h=M_X C_h M_X^{-1}$. By this matrix, points with a fixed parameter step $h$ are computed consequently according to Equation (\ref{Eqn:dynamic evaluation for general curve X(t)}). When a point $X(\theta)$ has been computed, the point $\textbf{r}(\theta)$ is obtained by choosing the first two coordinates.

\begin{table}[htbp]
\centering
\caption{Maximum deviations for evaluating the intrinsically defined curve.}
\begin{tabular}{|c|c|c|}
\hline
\#points & basis transformation & Taylor's method  \\
\hline
10      &  4.261E-14 & 2.916E+1  \\
\hline
20      &  5.153E-14 & 3.026E-1  \\
\hline
100     &  1.196E-13 & 1.971E-5  \\
\hline
200     &  2.160E-13 & 3.087E-7  \\
\hline
1000    &  6.407E-13 & 1.967E-11  \\
\hline
2000    &  1.467E-12 & 3.884E-13  \\
\hline
10000   &  4.606E-12 & 1.125E-12  \\
\hline
20000   &  3.954E-12 & 1.306E-11  \\
\hline
\end{tabular}\label{Table:Maximum errors for intrinsically defined curve}
\end{table}

In our experiments, we compute points on the curve that is defined on the parameter interval $[0,8\pi]$. Particularly, the start point is obtained as $X(0)=X_0+X_1$. The parameter step is chosen as $h=\frac{8\pi}{m}$ when $m$ points are to be computed on the curve. As $\rho(\theta)$ increases when $\theta$ changes from 0 to $8\pi$, the deviations from the evaluated points to the exact ones may increase too. To measure the accuracy of the proposed evaluation method, we compute the Euclidean distance from the last evaluated point to point $\mathbf{r}(8\pi)$ which is computed by loading the math library. From~\cite{Yang&Hong2017:DynamicEvaluation} we know that the lifted curve $X(\theta)$ satisfies a linear differential system and the points on the curve can be evaluated by employing Taylor's method or implicit mid-point scheme to solve the differential system. Since the implicit mid-point scheme has only quadratic precision for solving a linear differential system, we compare the results by the proposed method only with Taylor's method. The maximum deviations for evaluating points by basis transformation or by Taylor's method with various choices of number $m$ are given in Table~\ref{Table:Maximum errors for intrinsically defined curve}. From the table we can see that the evaluation accuracy by Taylor's method (expansion order $s=6$) depends heavily on the parameter steps while the presented basis transformation approach can always give high accuracy results for various choices of point numbers.

\textbf{Example 2.}
The second example is about dynamic evaluation of a B\'{e}zier curve with fixed or changing parameter steps.

\begin{figure}[htbp]
  \centering
  \subfigure[]{\includegraphics[width=5.2cm]{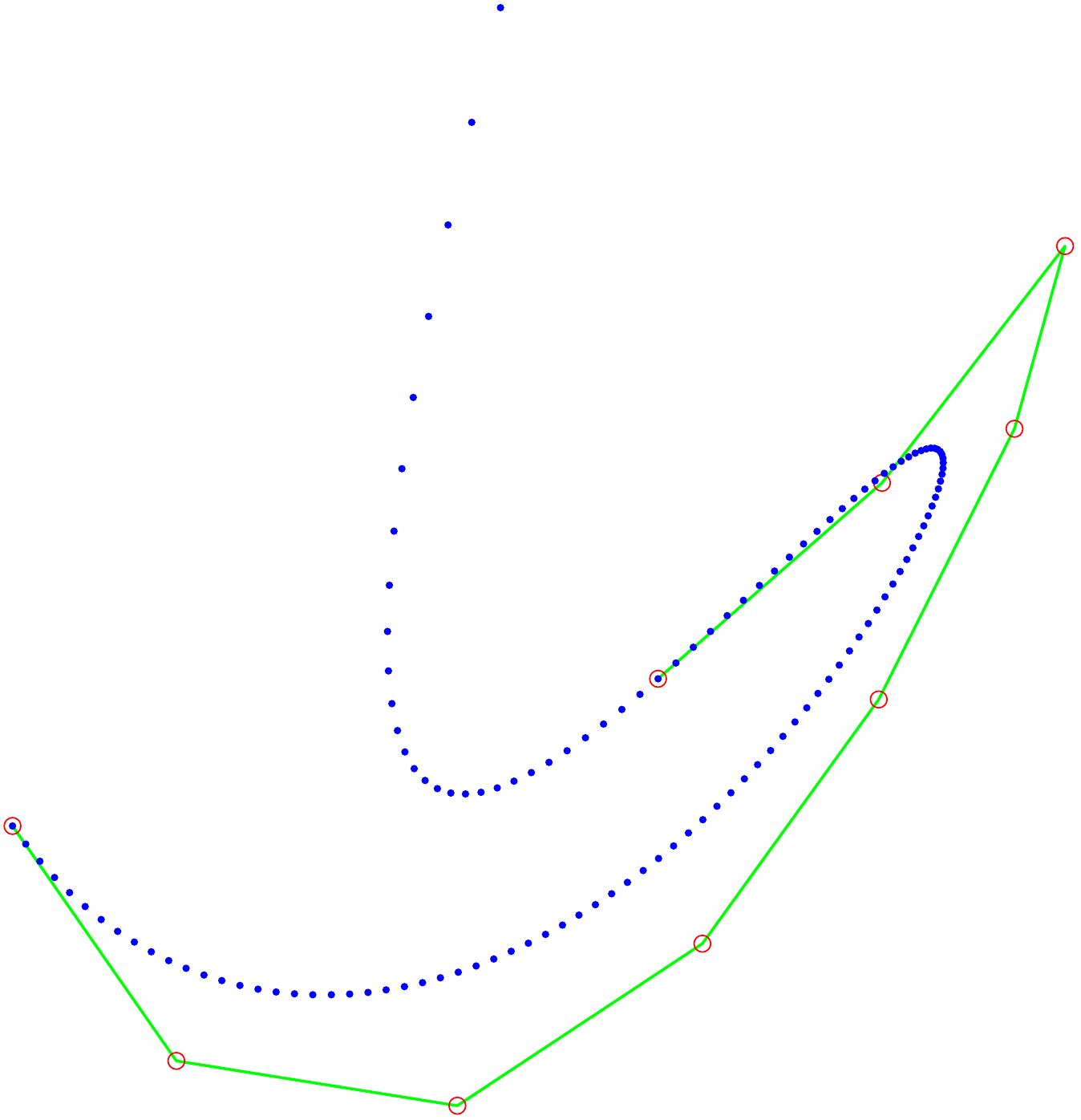}}
  \subfigure[]{\includegraphics[width=5.2cm]{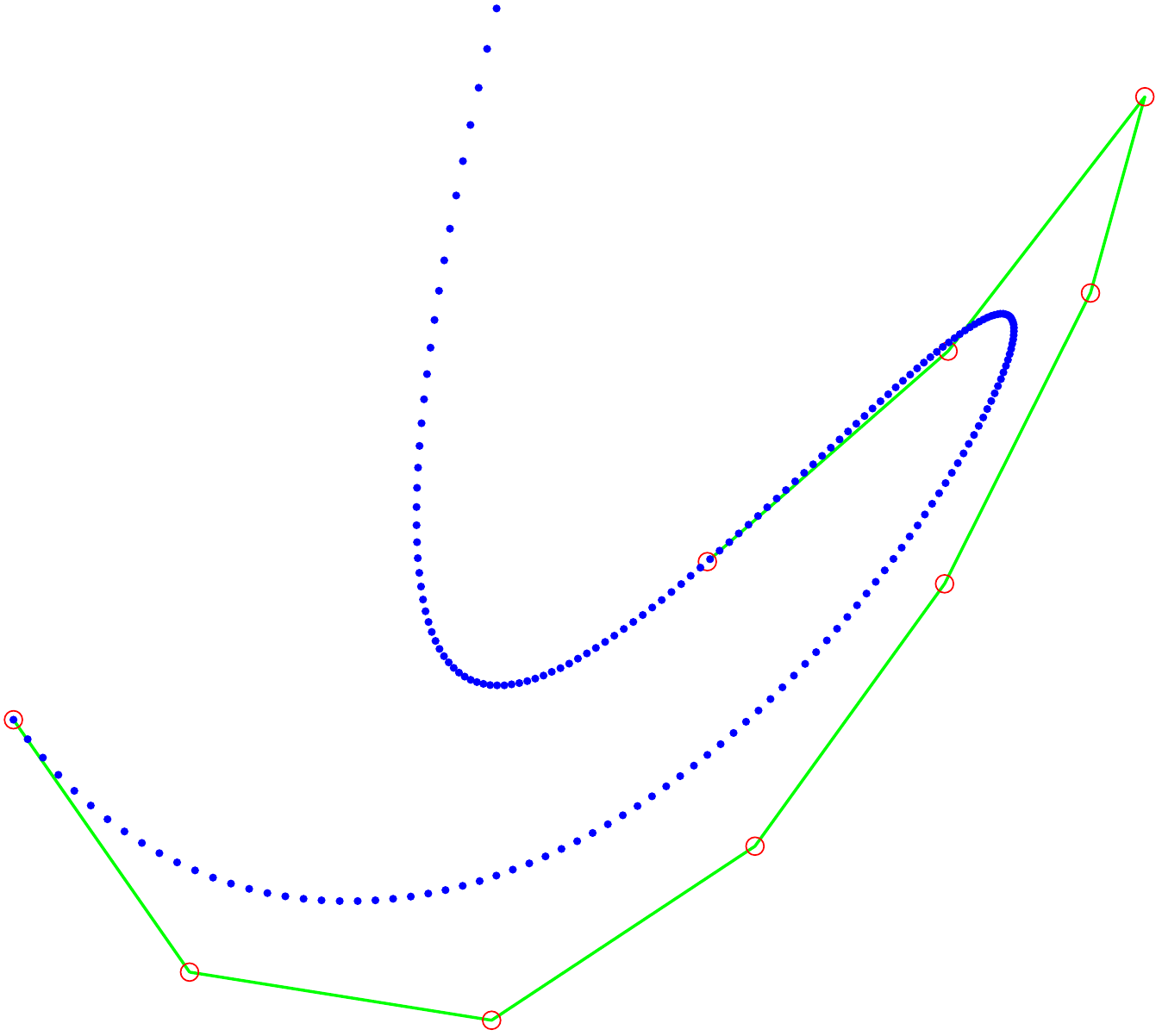}}\\
  \subfigure[]{\includegraphics[width=5.2cm]{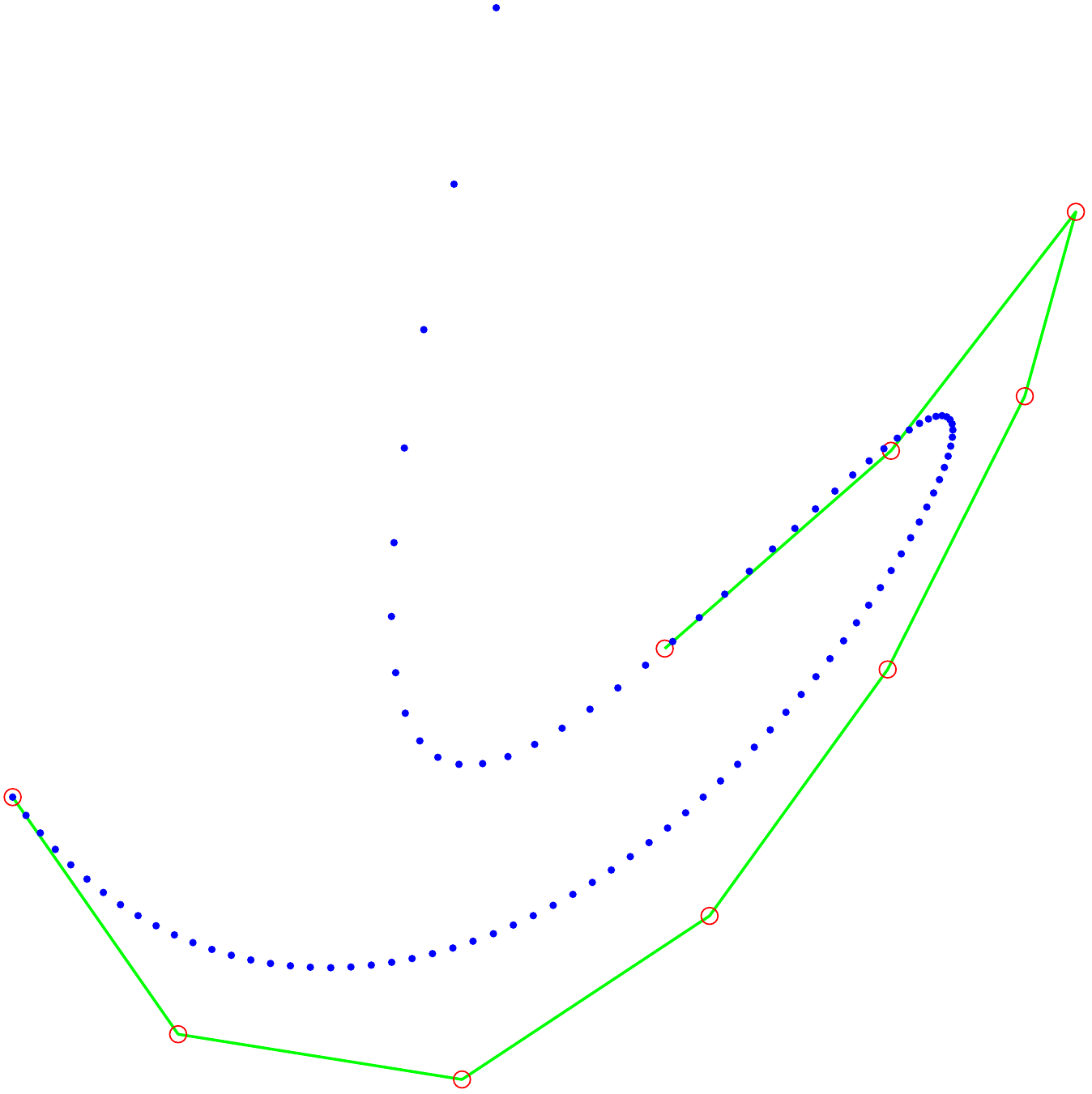}}
  \subfigure[]{\includegraphics[width=5.4cm]{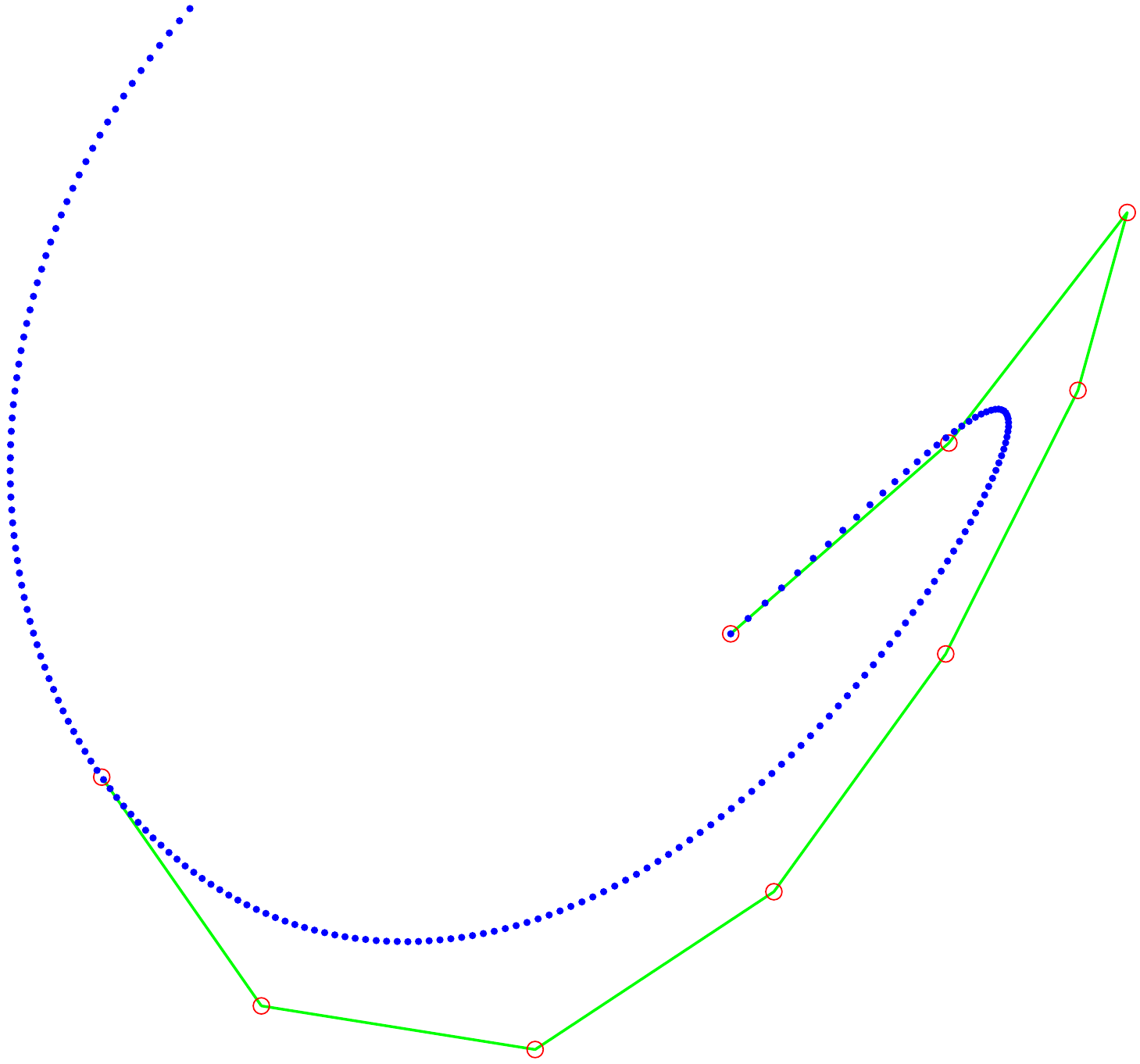}}
  \caption{Dynamic evaluation of a B\'{e}zier curve with fixed or changing parameter steps. (a) $a=0.01,b=1.01$; (b) $a=0.01,b=1.005$; (c) $a=0.01,b=1.015$; (d) $a=-0.005,b=0.99$. }
  \label{Fig:Bezier_curve}
\end{figure}

Figure~\ref{Fig:Bezier_curve} illustrates a planar B\'{e}zier curve of degree 8. To evaluate the curve by the proposed dynamic algorithm, we lift the curve from $\mathbb{R}^2$ to $\mathbb{R}^9$. Assume the lifted B\'{e}zier curve is $X(t)=\sum_{i=0}^8X_iB_{i,8}(t)$. It can then be represented as $X(t)=M_X \Phi_B(t)$, where $M_X$ and $\Phi_B(t)$ are the coefficient matrix or the basis vector, respectively. For any two distinctive real numbers $a$ and $b$, we compute a basis transformation matrix $C_B$ by Equation (\ref{Eqn:Bernstein basis transform}) and then a curve transformation matrix $M_B$ by Equation (\ref{Eqn:dynamic evaluation of Bezier curve}). By matrix $M_B$ we compute points in $\mathbb{R}^9$ and obtain points in the plane.

We first compute points by choosing $X(t_0)=X_0$, $a=0.01$ and $b=1.01$. The obtained points are plotted in Figure~\ref{Fig:Bezier_curve}(a). Because $b-a=1$, a sequence of points with a fixed forward parameter step have been obtained. By choosing $a=0.01$ and $b=1.005$, a sequence of points with decreasing parameter steps have been obtained by applying Equation (\ref{Eqn:dynamic evaluation of Bezier curve}) recursively. See Figure~\ref{Fig:Bezier_curve}(b) for the computed points starting from $X_0$. Similarly, points with increasing parameter steps can be obtained when we choose $a=0.01$ and $b=1.015$. See Figure~\ref{Fig:Bezier_curve}(c) for the evaluated points. When we choose $a=-0.005$, $b=0.99$ and $X(t_0)=X_8$, a sequence of points with decreasing parameter steps can be computed starting from the right end point of the curve. See Figure~\ref{Fig:Bezier_curve}(d).

From the above results we can see that different basis transformations can lead to different sampling speeds or directions on the curve. One can then tune the sampling speed or sampling direction adaptively by choosing various transformation matrices. Because a linear differential system can be constructed from each given B\'{e}zier curve~\cite{Yang&Hong2017:DynamicEvaluation}, derivatives at the sampled points can be evaluated directly using the differential system. As analyzed in~\cite{Yang&Hong2017:DynamicEvaluation}, dynamic evaluation of B\'{e}zier curves needs much less time than evaluating points individually using classical de Casteljau algorithm, even though both of the two algorithms have $O(n^2)$ time complexity.

\textbf{Example 3.}
In the third example we show how to evaluate piecewise smooth curves on a tensor product B\'{e}zier surface by the proposed algorithm.

Assume a B\'{e}zier surface of degree $5\times7$ is given by
\[
S(u,v)=\sum_{i=0}^5\sum_{j=0}^7 P_{ij}B_{i,5}(u)B_{j,7}(v), \ \ \ \ \ (u,v)\in[0,1]^2,
\]
where $P_{ij}\in \mathbb{R}^3$ are the control points. To evaluate points on the surface we first reformulate the surface in matrix form
\[
S(u,v)=\left(
            \begin{array}{ccc}
            B_{0,5}(u) & \cdots & B_{5,5}(u)
            \end{array}
        \right)
        \left(\begin{array}{ccc}
                             P_{0,0} & \cdots & P_{0,7} \\
                             \vdots  & \ddots & \vdots \\
                             P_{5,0} & \cdots & P_{5,7}
                           \end{array}
                        \right)
                        \left(\begin{array}{c}
                             B_{0,7}(v) \\
                             \vdots \\
                             B_{7,7}(v)
                           \end{array}
\right).
\]
Let $M_L=(P_{0,0},\ldots,P_{0,7},\ldots,P_{5,0},\ldots,P_{5,7})$, $\Phi_{B_n}(t)=(B_{0,n}(t),\ldots,B_{n,n}(t))^T$ and $\Phi(u,v)=\Phi_{B_5}(u)\otimes\Phi_{B_7}(v)$. The B\'{e}zier surface can be further represented as $S(u,v)=M_L\Phi(u,v)$. As $M_L$ is a $3\times48$ matrix, using the technique presented in~\cite{Yang&Hong2017:DynamicEvaluation}, we lift the B\'{e}zier surface from $\mathbb{R}^3$ to $\mathbb{R}^{48}$. Assume $M_X$ is the lifted matrix of order 48 and $X_{ij}$ are the lifted control points, the lifted surface becomes $X(u,v)=M_X\Phi(u,v)$. When $X(u,v)$ has been evaluated, the point $S(u,v)$ is obtained by choosing the first three coordinates of $X(u,v)$.

To derive the matrices for dynamic evaluation of curves on the B\'{e}zier surface, we first compute the transformation matrices for bases $\Phi_{B_5}(u)$ and $\Phi_{B_7}(v)$. Assume the translated parameter step is $h$, we choose $a=h$ and $b=1+h$. From Equation (\ref{Eqn:Bernstein basis transform}) we have the transformation matrices $C_{B_5}$ or $C_{B_7}$ for the basis $\Phi_{B_5}(u)$ or $\Phi_{B_7}(v)$. Then, the transformation matrices for the basis $\Phi(u,v)$ with translated parameter $u$ or $v$ are obtained as $C_u^h=C_{B_5}\otimes I_8$ or $C_v^h=I_6\otimes C_{B_7}$, respectively. Now, the matrices for dynamic evaluation of points on $u$-curves or $v$-curves on the lifted surface $X(u,v)$ with a fixed parameter step $h$ are obtained as $M_u^h=M_X C_u^h M_X^{-1}$ and $M_v^h=M_X C_v^h M_X^{-1}$. If we replace the parameter step $h$ by $-h$, we have matrices $M_u^{-h}$ and $M_v^{-h}$ for dynamic evaluation of $u$-curves or $v$-curves in the opposite directions.

\begin{figure}[htbp]
  \centering
  \subfigure[]{\includegraphics[width=6.0cm]{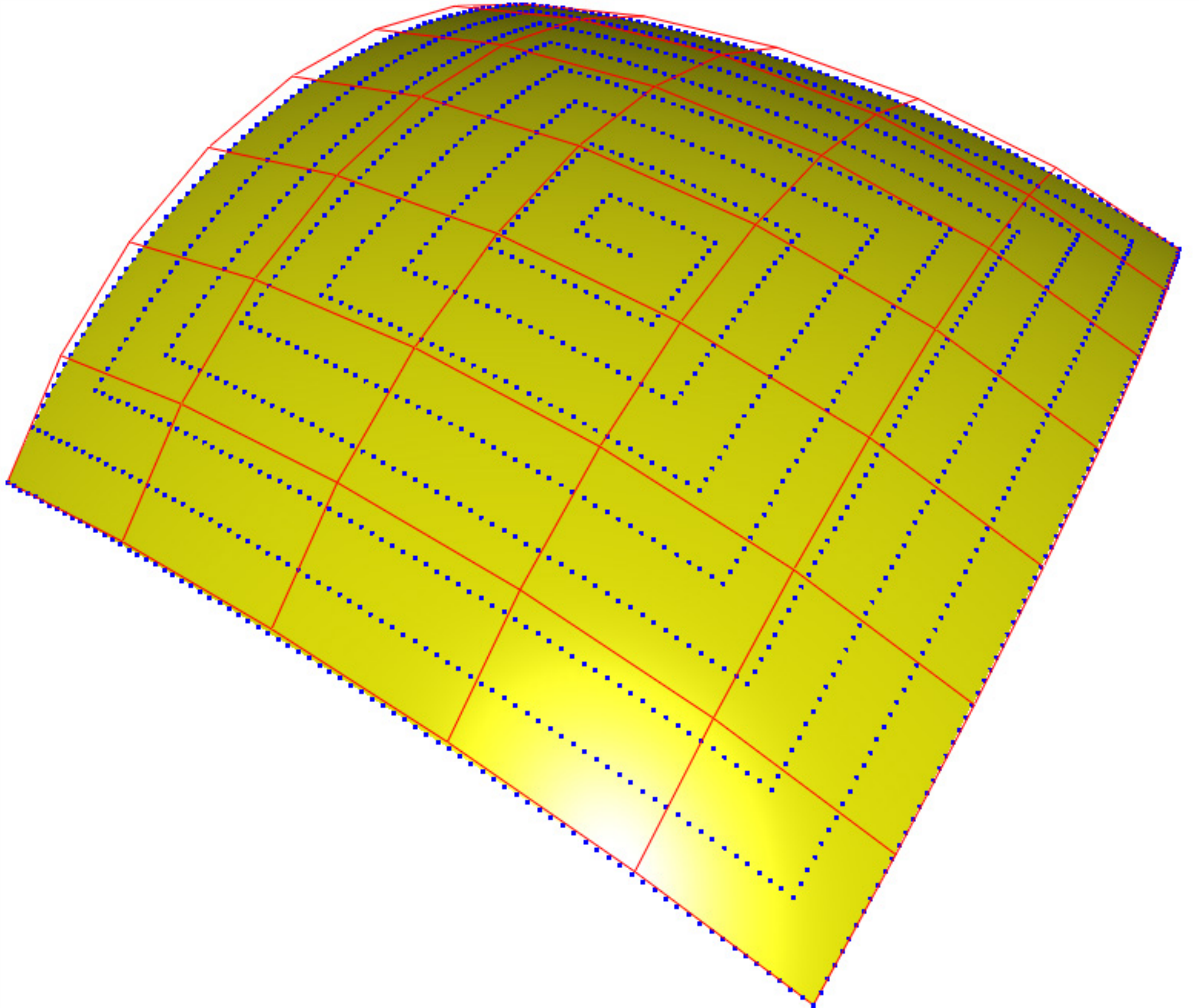}}
  \subfigure[]{\includegraphics[width=6.0cm]{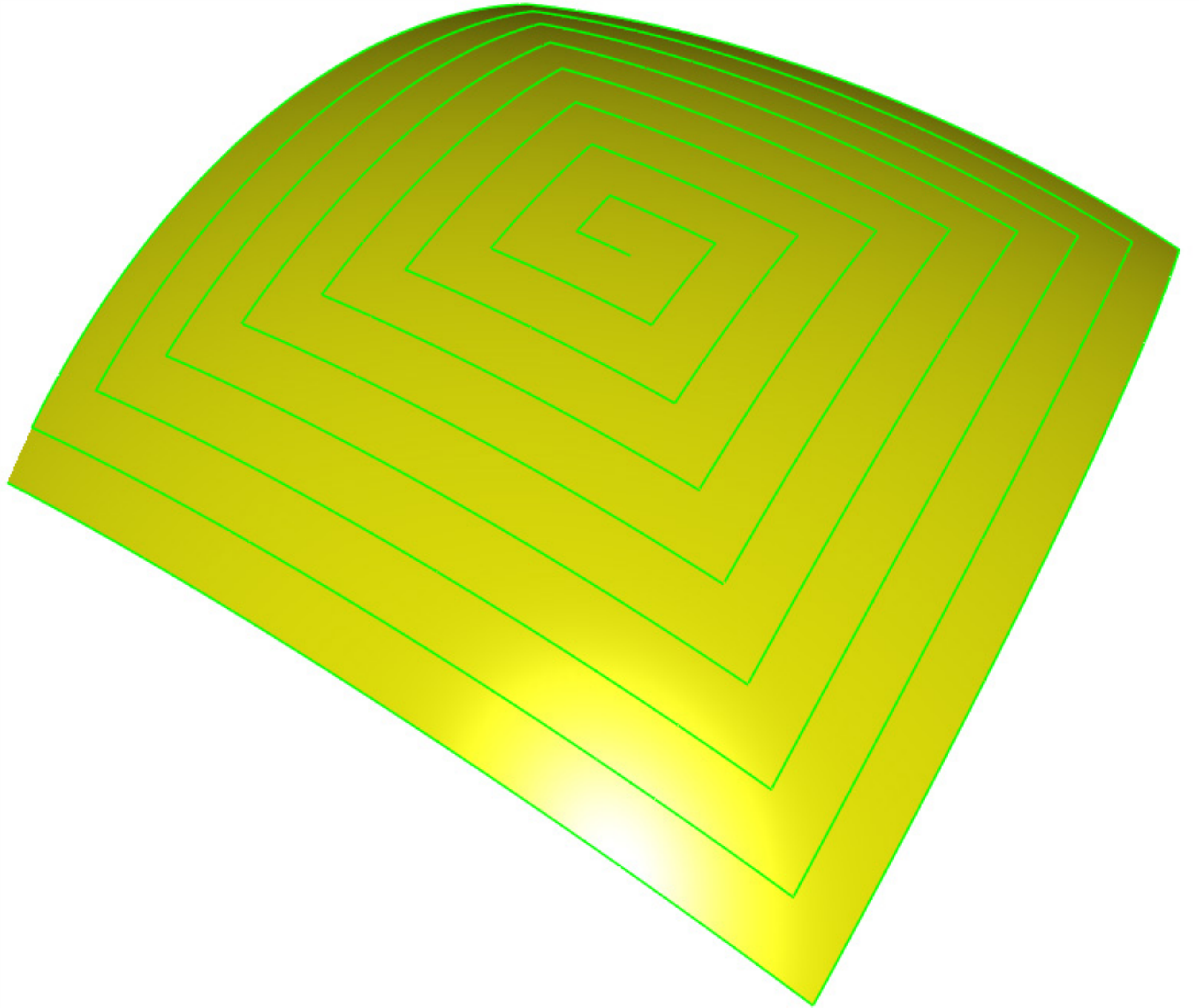}}
  \caption{Dynamic evaluation of curves on a B\'{e}zier patch: (a) the tensor-product B\'{e}zier patch and the evaluated points; (b) the curves generated by the evaluated points. }
  \label{Fig:Bezier patch}
\end{figure}

In our experiments, we choose $h=\frac{1}{80}$ and $X_{0,0}$ as the initial point for dynamic evaluation of a piecewise smooth curve that is consisting of 33 pieces of full or partial iso-parameter curves. Particularly, we evaluate points on $u$-curves with parameter step $h$, $v$-curves with parameter step $h$, $u$-curves with parameter step $-h$ and $v$-curves with parameter step $-h$, alternately. Assume the curve segments are numbered as $j=0,1,\ldots,32$. The point number for each curve segment is chosen as $m_j=80-5\times[(j-1)/2]$, where $[(j-1)/2]$ means the integer part of a real number. When points on a specified curve segment have been evaluated, the obtained last point is chosen as the start point for dynamic evaluation of next curve segment. See Figure~\ref{Fig:Bezier patch}(a) for the evaluated points and Figure~\ref{Fig:Bezier patch}(b) for the obtained piecewise curve. We note that the lastly evaluated point by the proposed technique is corresponding to the center of the surface. Assume the distance between two corner control points $P_{0,0}$ and $P_{5,7}$ is 1. The absolute error between the last point obtained by the proposed algorithm and $S(0.5,0.5)$ computed by conventional de Casteljau algorithm is $7.931\times 10^{-13}$.

\textbf{Example 4.}
In the fourth example we evaluate a family of iso-parameter curves on a helicoidal patch. Suppose the surface patch is given by
\begin{equation}
\left\{
\begin{array}{lcl}
x(u,v) &=& (2+u)\cos v, \\
y(u,v) &=& (2+u)\sin v, \ \ \ \ \ (u,v)\in[0,2]\times[0,4\pi]. \\
z(u,v) &=& v,
\end{array}
\right.
\label{Eqn:helicoidal patch}
\end{equation}
Let $\Phi(u,v)=(1,v,\cos v,\sin v,u\cos v,u\sin v)^T$. It is easily verified that the space spanned by the basis $\Phi(u,v)$ is closed with respect to partial differentiations $\frac{\partial}{\partial u}$ and $\frac{\partial}{\partial v}$. To evaluate the surface by the dynamic algorithm, we lift the surface from $\mathbb{R}^3$ to $\mathbb{R}^6$ by adding three more coordinates to the coefficients. The lifted surface represented in matrix form is
\[
X(u,v)=\left(\begin{array}{cccccc}
                             0 & 0 & 2 & 0 & 1 & 0 \\
                             0 & 0 & 0 & 2 & 0 & 1 \\
                             0 & 1 & 0 & 0 & 0 & 0 \\
                             1 & 0 & 0 & 0 & 0 & 0 \\
                             0 & 0 & 0 & 0 & 1 & 0 \\
                             0 & 0 & 0 & 0 & 0 & 1
                           \end{array}
                        \right)
                        \left(\begin{array}{c}
                             1 \\
                             v \\
                             \cos v \\
                             \sin v \\
                             u\cos v \\
                             u\sin v
                           \end{array}
\right).
\]
Denote the coefficient matrix of $X(u,v)$ as $M_X$. It yields that $X(u,v)=M_X\Phi(u,v)$. When $X(u,v)$ has been evaluated, the point $(x(u,v),y(u,v),z(u,v))$ is obtained by choosing the first three coordinates.

\begin{figure}[htbp]
  \centering
  \subfigure[]{\includegraphics[width=4cm]{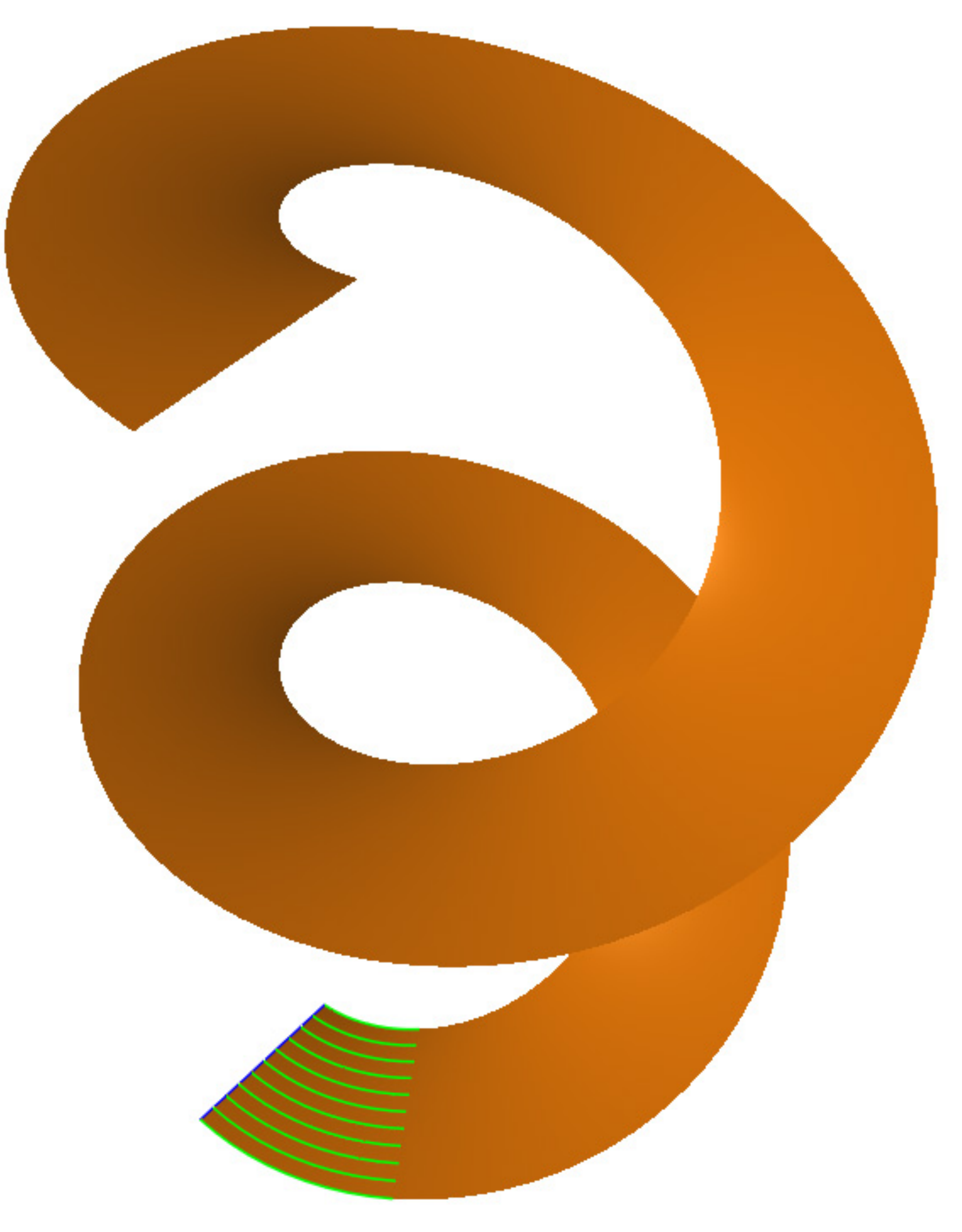}}
  \subfigure[]{\includegraphics[width=4cm]{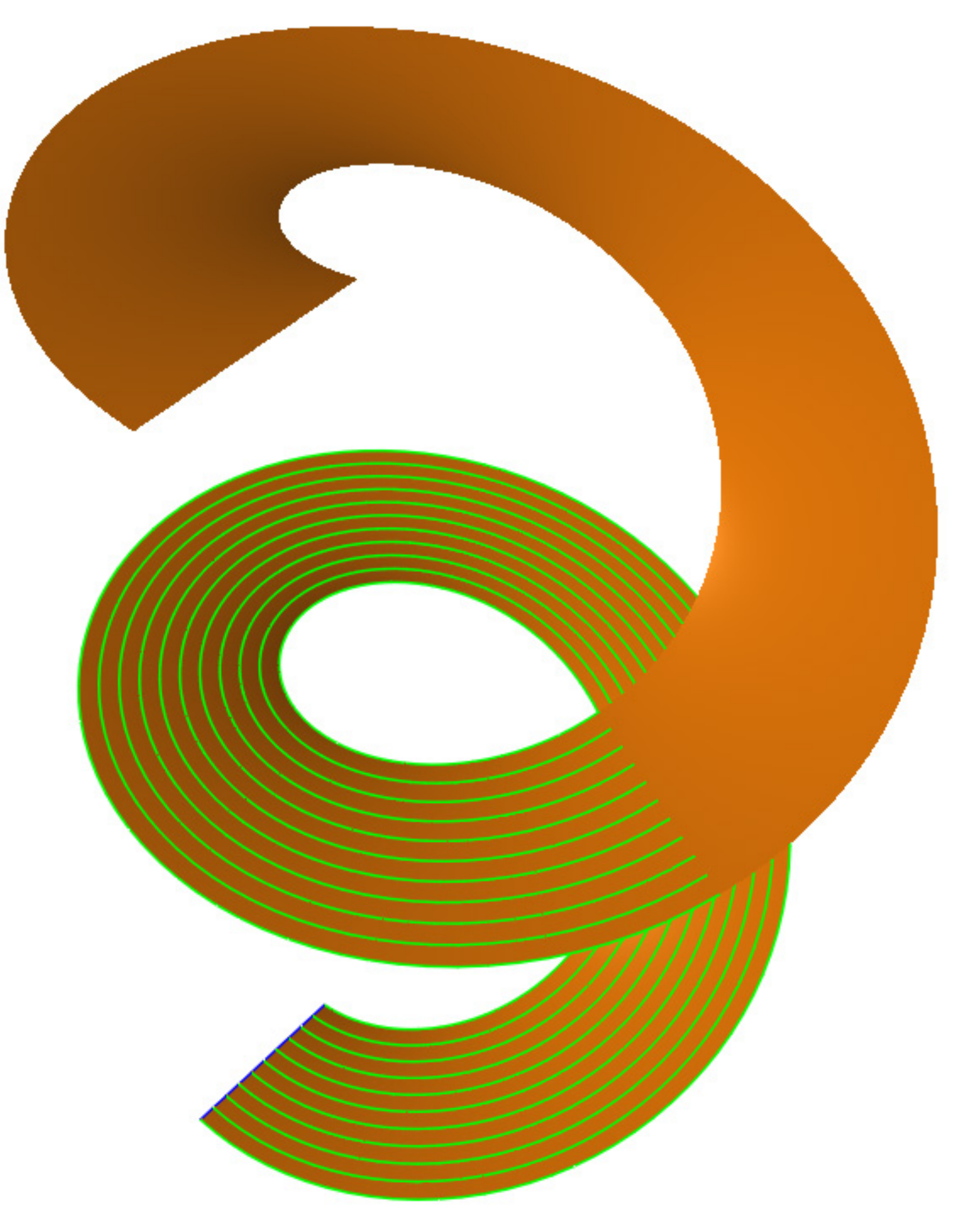}}
  \subfigure[]{\includegraphics[width=4cm]{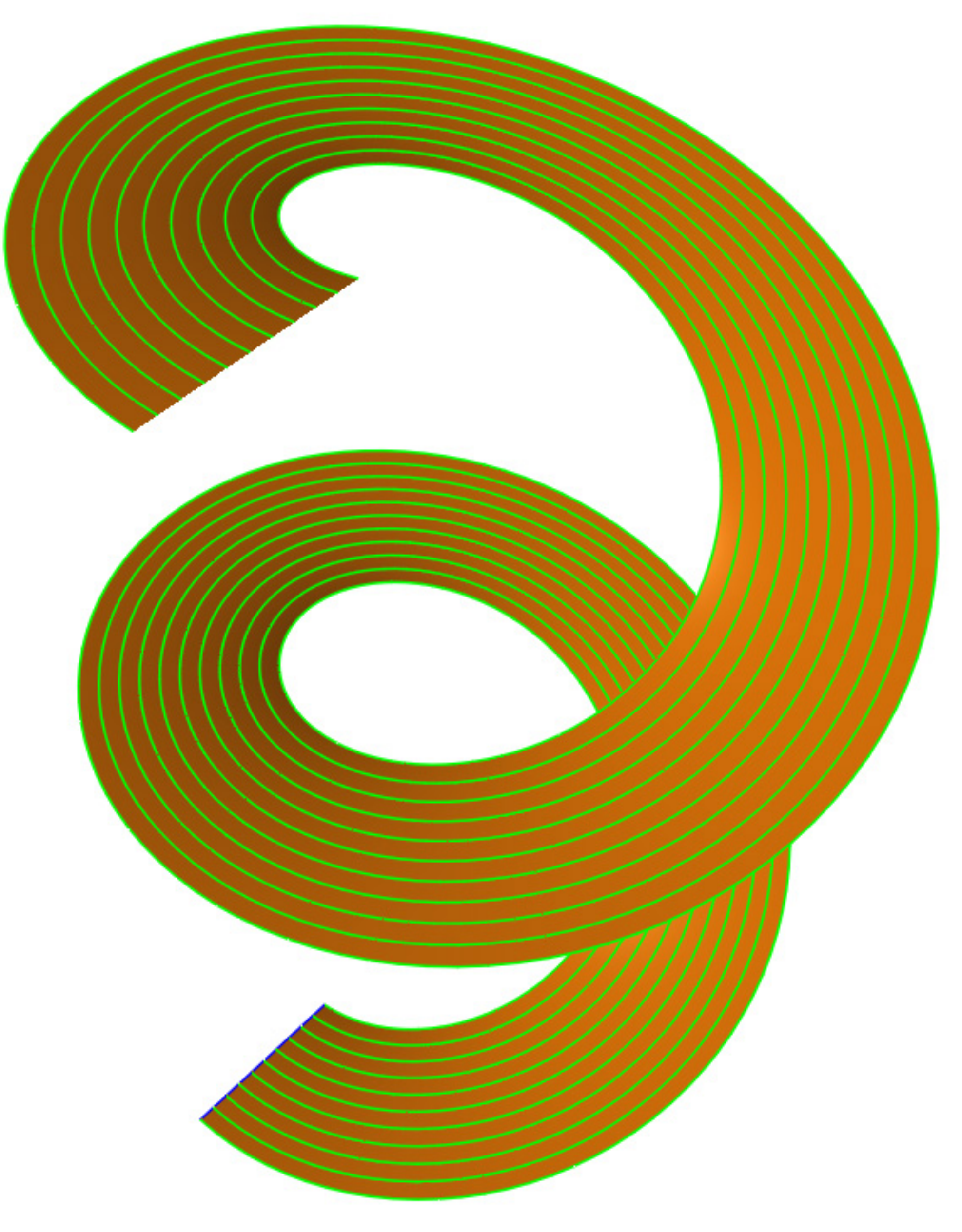}}
  \caption{Dynamic evaluation of a set of iso-parameter curves on the helicoid patch. The results are obtained by (a)10, (b)120 or (c)200 steps of evaluation.}
  \label{Fig:helicoid patch}
\end{figure}

Because the basis vector $\Phi(u,v)$ can be decomposed as $\Phi(u,v)=U_1(v) \sqcup U_1(u) \otimes V(v)$, the transformation matrix for the basis vector $\Phi(u,v)$ with respect to the translation of parameter $v$ is obtained as $C_v^h=\diag(M_{U_1}^h, I_2 \otimes M_V^h)$. Because $\det(M_X)=-4\neq 0$, we compute a transformation matrix as $M_v^h=M_X C_v^h M_X^{-1}$. The points on any surface curve with a fixed parameter $u$ are then computed by
\begin{equation}
\left\{
\begin{array}{lcl}
X(u,v+h) & = & M_v^hX(u,v), \\
X(u,0)   & = & \left(\begin{array}{c}
                             2+u \\
                             0 \\
                             0 \\
                             1 \\
                             u \\
                             0
                           \end{array}
               \right),
               \ \ \ \ \ u\in[0,2].
\end{array}
\right.
\label{Eqn:dynamic evaluation for X(u,v)}
\end{equation}
According to Equation (\ref{Eqn:dynamic evaluation for X(u,v)}), points on a family of $v$-curves on the surface are obtained iteratively starting from a set of points on the boundary line. Figure~\ref{Fig:helicoid patch} illustrates the evaluated results after 10, 120 or 200 steps of evaluation, where $u=0, 0.2, 0.4, \ldots, 2$ and the parameter step is chosen as $h=\frac{4\pi}{200}$.

\textbf{Example 5.}
Lastly, we evaluate curves with skew parametrization on a Dupin-Cyclide. Let $a=6$, $b=4\sqrt{2}$, $c=2$ and $\mu=3$. The Cartesian coordinates of a Dupin-Cyclide are given by~\cite{Roth2015:DescrptionCurvesInExtChebSpace}
\begin{equation}
\left\{
\begin{array}{lcl}
x(u,v) &=& \frac{x_1(u,v)}{x_4(u,v)},   \\
y(u,v) &=& \frac{x_2(u,v)}{x_4(u,v)}, \ \ \ \ \ (u,v)\in[0,2\pi]^2, \\
z(u,v) &=& \frac{x_3(u,v)}{x_4(u,v)},
\end{array}
\right.
\label{Eqn:Dupin_Cyclide_definition}
\end{equation}
where
\[
\left\{
\begin{array}{lcl}
x_1(u,v) &=& \mu(c-a\cos u\cos v)+b^2\cos u,   \\
x_2(u,v) &=& b\sin u(a-\mu\cos v),  \\
x_3(u,v) &=& b\sin v(c\cos u-\mu),  \\
x_4(u,v) &=& a-c\cos u\cos v.
\end{array}
\right.
\]
To evaluate the Cartesian coordinates of the surface, we should compute the homogeneous coordinates first.
Let
\[
M_H=
\left(\begin{array}{ccccccccc}
                             \mu c & 0 & 0 & 0 & 0 & 0 & b^2 & 0 & -\mu a \\
                             0 & 0 & 0 & ab & 0 & -\mu b & 0 & 0 & 0 \\
                             0 & -\mu b & 0 & 0 & 0 & 0 & 0 & bc & 0 \\
                             a & 0 & 0 & 0 & 0 & 0 & 0 & 0 & -c
                           \end{array}
                        \right),
\]
and $\Phi(u,v)=(1,\sin v,\cos v,\sin u,\sin u\sin v,\sin u\cos v,\cos u,\cos u\sin v,\cos u\cos v)^T$. The homogeneous coordinates of the surface are represented as $X_H(u,v)=M_H\Phi(u,v)$. To evaluate the homogeneous coordinates dynamically, we lift $X_H(u,v)$ from $\mathbb{R}^4$ to $\mathbb{R}^9$. Assume $M_H=(H_1,H_2)$, where $H_1$ and $H_2$ are the $4\times5$, $4\times4$ sub-matrices, respectively. Let
\[
M_X=\left(\begin{array}{cc}
            H_1 & H_2 \\
            I_5 & 0
           \end{array}
    \right),
\]
where $I_5$ is the identity matrix of order 5. The lifted homogeneous surface is obtained as $X(u,v)=M_X\Phi(u,v)$.

Let $\Theta(t)=(1,\sin t,\cos t)^T$. It yields that $\Phi(u,v)=\Theta(u)\otimes\Theta(v)$. It is easily verified that spaces spanned by $\Theta(t)$ or $\Phi(u,v)$ are closed with respect to a differentiation or translation of the parameters. We have $\Theta(t+h)=C_h\Theta(t)$ and $\Phi(u+h_1,v+h_2)=C_{u,v}^{h_1,h_2}\Phi(u,v)$, where
\[
C_h=
\left(\begin{array}{ccc}
            1   &   0       &   0       \\
            0   &   \cos h  &   \sin h \\
            0   &  -\sin h  &   \cos h
      \end{array}
\right)
\]
and $C_{u,v}^{h_1,h_2}=C_{h_1}\otimes C_{h_2}$. As the inverse of matrix $M_X$ is
\[
M_X^{-1}=
\left(\begin{array}{cc}
            0       &   I_5        \\
            H_2^{-1}  &  -H_2^{-1}H_1
      \end{array}
\right),
\]
we have $X(u+h_1,v+h_2)=M_{u,v}^{h_1,h_2}X(u,v)$, where $M_{u,v}^{h_1,h_2}=M_X C_{u,v}^{h_1,h_2} M_X^{-1}$. Starting from any point $X(u_0,v_0)$, a sequence of points on the surface will be computed by
\begin{equation}
\left\{
\begin{array}{lcl}
X(u+h_1,v+h_2) & = & M_{u,v}^{h_1,h_2}X(u,v), \\
X(u_0,v_0)     & = & M_X\Phi(u_0,v_0).
\end{array}
\right.
\label{Eqn:Dynamic evaluation of Dupin_Cyclide}
\end{equation}

\begin{figure}[htbp]
  \centering
  \subfigure[]{\includegraphics[width=6cm]{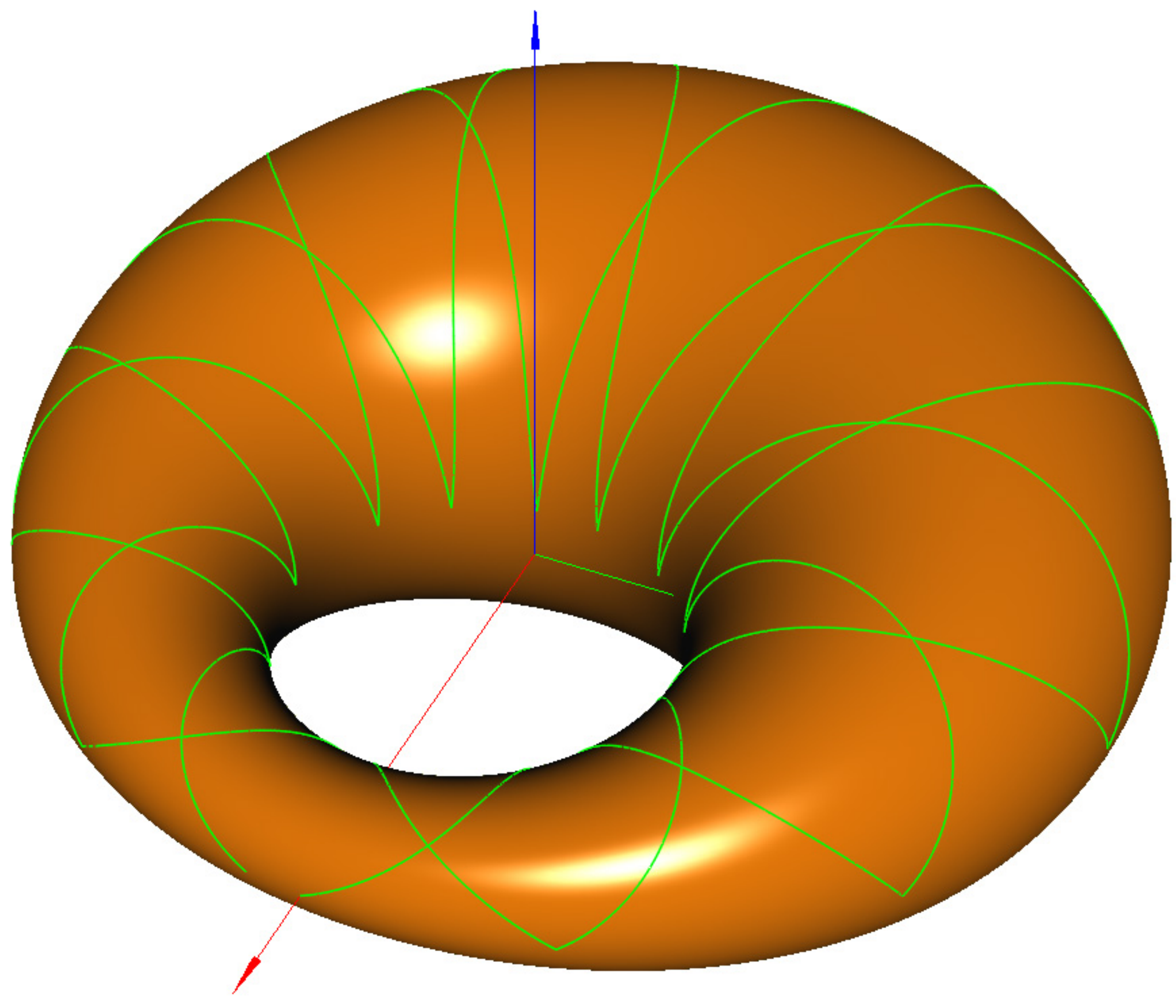}}
  \subfigure[]{\includegraphics[width=6cm]{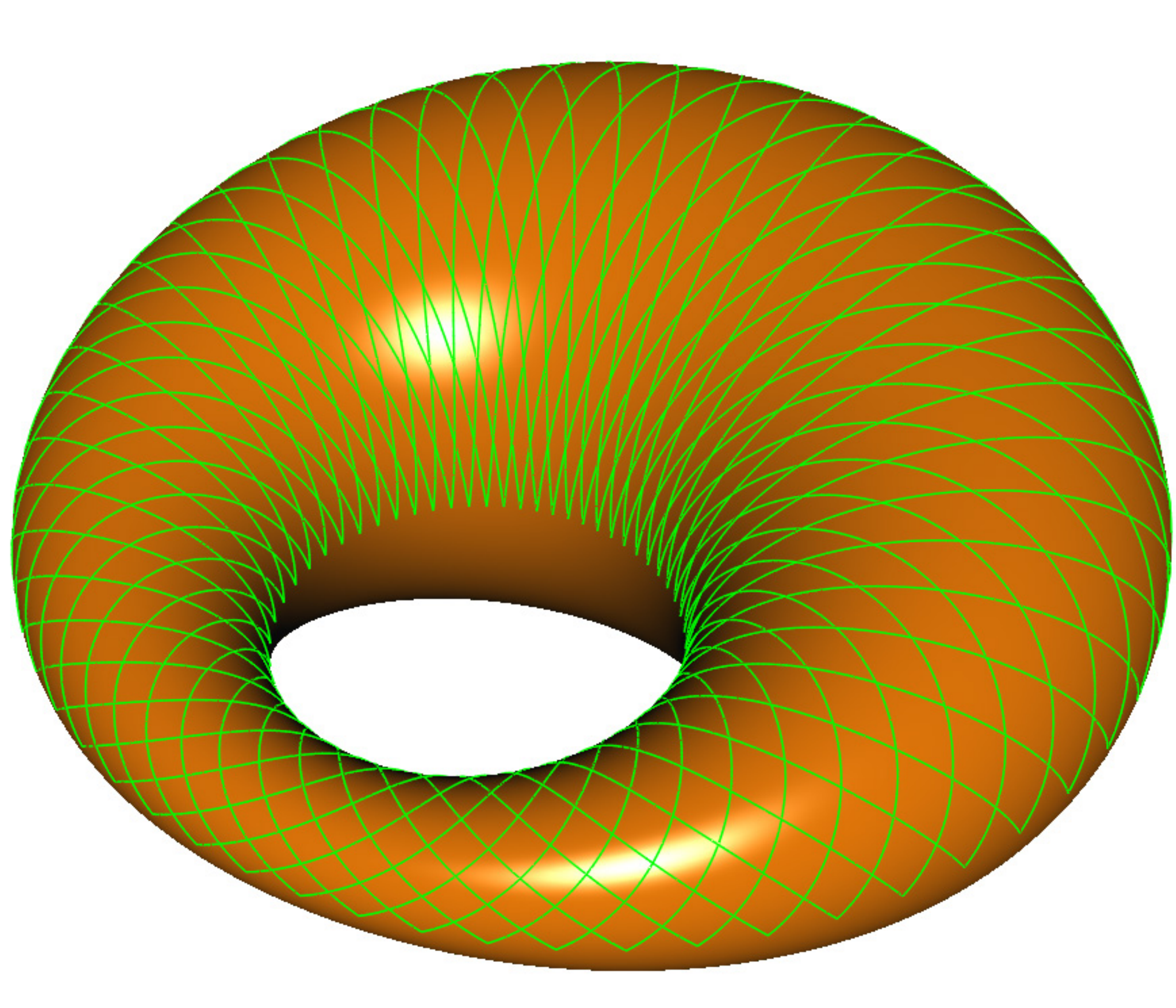}}
  \caption{Dynamic generation of piecewise curves with skew parametrization on the Dupin-Cyclide: (a) 22 pieces; (b) 100 pieces.  }
  \label{Fig:Dupin Cyclide}
\end{figure}

In our experiments, we choose $u_0=0$, $v_0=\pi$ and the start point is obtained as the outer intersection point between the surface and the $x$-axis (the red arrow line) in the positive direction; see Figure~\ref{Fig:Dupin Cyclide}(a). To evaluate points on piecewise surface curves with skew parametrization, we choose the parameter steps as $h_1=\frac{0.18\pi}{m}$ and $h_2=\frac{\pi}{m}$, where $m$ is the number of points that will be computed on each piece of surface curve. By choosing $m=100$, two iteration matrices $M_{u,v}^{h_1,h_2}$ and $M_{u,v}^{h_1,-h_2}$ are computed first. Given $X(u_0,v_0)$, the points on the first piece of surface curve are computed by Equation (\ref{Eqn:Dynamic evaluation of Dupin_Cyclide}) using matrix $M_{u,v}^{h_1,h_2}$. Starting from the end of first piece of curve we compute points on the second piece using matrix $M_{u,v}^{h_1,-h_2}$. We continue this process by using $M_{u,v}^{h_1,h_2}$ and $M_{u,v}^{h_1,-h_2}$ for evaluating odd number or even number pieces of curves alternately. Figure~\ref{Fig:Dupin Cyclide}(a) illustrates the result with 22 pieces of surface curves while Figure~\ref{Fig:Dupin Cyclide}(b) illustrates the surface curve with 100 pieces. Due to periodicity, the last point on the 100th piece is theoretically the same as the start point of the first one. Practically, the distance between these two points is $6.526\times 10^{-13}$ even after 10000 times of matrix-point multiplication in the presence of truncation errors of irrational numbers.

\textbf{Numerical stability of dynamic evaluation.}
Since all dynamically evaluated points on a curve or surface are computed by one or a few constant iteration matrices and a start point, the accuracy of the constant iteration matrices play a key role for the accuracy of the evaluated points.
The results computed by the Taylor method in Example 1 show that an inexact iteration matrix may cause deviations in the following evaluated points. All examples employing the basis transformation technique demonstrate that the new method can be used to compute the iteration matrices and the points on curves or surfaces accurately enough.

Even the iteration matrix is accurate, the noise at the initial point can propagate to the dynamically evaluated points.
Fortunately, the propagated errors are bounded and controlled when points on a curve segment or a surface patch are dynamically evaluated. Suppose $X(t)=M_X\Phi(t)$ is an exponential polynomial curve as defined in Section~\ref{Subsection:evaluation of exponential polynomial curves} and a sequence of points on the curve are computed by Equation (\ref{Eqn:dynamic evaluation for general curve X(t)}). If the start point has been changed as $\tilde{X}(t_0)=X(t_0)+X_\varepsilon$, the dynamically evaluated points become
\[
\tilde{X}(t_i)=M_h \tilde{X}(t_{i-1}) = M_h^i (X(t_0)+X_\varepsilon),
\]
where $M_h=M_XC_hM_X^{-1}$ and $C_h=e^{Ah}$ are as defined in Equation (\ref{Eqn:dynamic evaluation for general curve X(t)}) or Proposition~\ref{Proposition:space with tranlated basis}. The error magnitude for the $i$th point is estimated as
\[
||\tilde{X}(t_i)-X(t_i)||_\infty = ||M_h^i X_\varepsilon||_\infty \leq ||M_X e^{Aih} M_X^{-1}||_\infty ||X_\varepsilon||_\infty
\]
Since we compute points on a curve segment or a surface patch in practice, the parameter $ih$ lies on a limited interval. Therefore, the norm of the matrix $M_X e^{Aih} M_X^{-1}$ and the noise magnitudes of the evaluated points are bounded.
We have recomputed points for above examples using start points with added noise. It is found that the deviation magnitudes for the dynamically evaluated points are around the same or a few times larger than the magnitudes of added noise.

%%%%%%%%%%%%%%%%%%%%%%%%%%%%%%%%%%%%%%%%%%%%%%%%%%%%%%%%%%%%%%%%%%%%%%%%%%%%%
%%                                                                          %
%% Section 5                                                                %
%%                                                                          %
%%%%%%%%%%%%%%%%%%%%%%%%%%%%%%%%%%%%%%%%%%%%%%%%%%%%%%%%%%%%%%%%%%%%%%%%%%%%%

\section{Conclusions}
\label{Sec:Conclu}
This paper has presented a robust and efficient algorithm for dynamic evaluation of free-form curves and surfaces constructed by general exponential polynomials. By explicit computation of transformation matrices between exponential polynomial bases with or without translation of the parameter, points on curves or surfaces with equal parameter steps can be evaluated dynamically with only arithmetic operations. The proposed technique suffers no shortcomings of classical numerical algorithms for solving linear differential systems any more and it can be used for accurate and  stable evaluation of general exponential polynomial curves with any parameter steps. Besides evaluating points with fixed parameter steps or families of iso-parameter curves on surfaces, the basis transformation technique can also be used for evaluating polynomial curves with changing parameter steps or dynamic evaluation of skew-parameterized curves on surfaces in a simple and efficient way.

\section*{Acknowledgment}
We owe thanks to referees for their invaluable comments and suggestions which helped to improve the presentation of the paper greatly.

\bibliographystyle{siamplain}
\bibliography{DynamicBasisTransform}

\end{document}